\documentclass[10pt,a4paper]{article}

\usepackage[a4paper, margin=1in]{geometry} 

\usepackage[utf8]{inputenc}
\usepackage[T1]{fontenc}
\usepackage{algorithmic}
\usepackage{amsfonts}
\usepackage{amsmath}
\usepackage{amssymb} 
\usepackage{amsthm}
\usepackage{anyfontsize}
\usepackage{booktabs} 
\usepackage{caption} 
\usepackage{color}
\usepackage{comment}
\usepackage{courier} 
\usepackage{dsfont}
\usepackage{float}
\usepackage{graphicx} 
\usepackage{latexsym}
\usepackage{mathtools}
\usepackage{qtree}
\usepackage{slashed} 
\usepackage{subcaption}
\usepackage{tikz-cd}
\usepackage{url}
\usepackage{verbatim}
\usepackage{xcolor}
\usepackage{xurl}

\usepackage{hyperref}
\definecolor{darkblue}{rgb}{0, 0, 0.5}
\hypersetup{colorlinks=true,citecolor=darkblue, linkcolor=darkblue, urlcolor=darkblue}

\usepackage{harpoon} 

\usepackage[style=numeric, 
natbib=true, 
backend=biber, 
giveninits=true, 
uniquename=init,
isbn=false,
doi=false,
eprint=false,
maxbibnames=3, 
maxcitenames=2
]{biblatex}
\usepackage{soul} 

\DeclareMathOperator{\Tr}{{Tr}}

\DeclareMathOperator{\Vect}{{Vect}}

\DeclareMathOperator{\Aut}{{Aut}}

\DeclareMathOperator{\Symm}{{Symm}}

\DeclareMathOperator{\D}{{d}}

\DeclareMathOperator{\End}{{End}}
\DeclareMathOperator{\Fock}{{\Gamma}}

\DeclareMathOperator{\Hom}{{\textsc{Hom}}}

\newcommand{\djunion}{\sqcup}

\DeclareMathOperator{\Dom}{{Dom}}

\newcommand{\Endo}[1]{\textsc{End}(#1)}
\newcommand{\Linspan}[1]{\textsc{Span}\{#1\}}

\newcommand{\tensor}{\otimes}
\newcommand{\tsr}{\otimes}
\newcommand{\algtsr}{\otimes_{\txt{alg}}}

\newcommand{\closure}[1]{\overline{#1}}

\newcommand{\ip}[1]{{\langle #1\rangle}}

\newcommand{\NN}{{\mathds{N}}}

\newcommand{\laplace}{\mathop{}\!\mathbin\bigtriangleup}
\newcommand{\diracop}{\slashed{D}}
\newcommand{\dlaplace}{\diracop^2}

\newcommand{\grad}{\nabla}

\newcommand{\footurl}[1]{{\footnote{\href[#1]{#1}}}}

\newcommand{\Rnonneg}{{\mathds{R}_{\geq 0}}}

\newcommand{\one}{{\textbf{1}}}

\newcommand{\norm}[1]{\left\lVert#1\right\rVert}
\newcommand{\vertiii}[1]{{\left\vert\kern-0.25ex\left\vert\kern-0.25ex\left\vert #1 
    \right\vert\kern-0.25ex\right\vert\kern-0.25ex\right\vert}}

\newtheorem{theorem}{Theorem}

\newtheorem{lemma}[theorem]{Lemma}

\newtheorem{corollary}[theorem]{Corollary}

\theoremstyle{definition}

\newtheorem{definition}{Definition}
\newtheorem{thm}{Theorem}
\newtheorem{prop}{Proposition}
\newtheorem{remark}{Remark}

\newtheorem{example}{Example}
\newtheorem{observation}{Observation}

\newcommand{\CStar}{\emph{C}\textsuperscript{*}}

\newcommand{\Gc}{{\mathcal{G}}}
\newcommand{\Fc}{{\mathcal{F}}}

\newcommand{\Ec}{{\mathcal{E}}}

\newcommand{\Hc}{{\mathcal{H}}}
\newcommand{\Ac}{{\mathcal{A}}}
\newcommand{\Bc}{{\mathcal{B}}}

\newcommand{\Zc}{{\mathcal{Z}}}
\newcommand{\Sc}{{\mathcal{S}}}
\newcommand{\Nc}{{\mathcal{N}}}
\newcommand{\Mc}{{\mathcal{M}}}
\newcommand{\Oc}{{\mathcal{O}}}
\newcommand{\Wc}{{\mathcal{W}}}

\newcommand{\Lc}{{\mathcal{L}}}

\newcommand{\C}{{\mathcal{C}}}

\newcommand{\Hs}{{\textsf{H}}}

\newcommand{\Vs}{{\textsf{V}}}

\newcommand{\Ls}{{\textsf{L}}}

\newcommand{\CC}{\mathbb{C}}

\newcommand{\EE}{\mathbb{E}}

\newcommand{\KK}{\mathbb{K}}

\newcommand{\PP}{\mathbb{P}}

\newcommand{\RR}{\mathds{R}}

\newcommand{\ZZ}{\mathds{Z}}

\newcommand{\Lie}[1]{\textsc{Lie}[{#1}]}

\newcommand{\txt}[1]{\text{#1}}

\newcommand{\txts}[1]{\textsc{#1}}

\newcommand{\aln}[1]{\begin{align*}#1\end{align*}}

\newcommand{\isomorph}{\cong}

\newcommand{\mat}{\textsc{Mat}}

\newcommand{\namecite}[1]{\citet{#1}}

\newcommand{\map}[1]{ \xrightarrow{#1}  }

\DeclareMathOperator{\ClifB}{{Cl}} 
\DeclareMathOperator{\ClifA}{{Cl}} 
\newcommand{\Spin}{\textsc{Spin}} 
\newcommand{\MSect}{\Gamma} 
\newcommand{\Sect}[1]{\Gamma(#1)} 

\newcommand{\SmoothSect}[1]{\Gamma^\infty(#1)}

\newcommand{\ExpVec}[1]{\Ec(#1)}

\newcommand{\dVol}{\D_{vol}}
\newcommand{\I}{{\mathrm{i}}}
\newcommand{\Rm}{{\txts{Rm}}}

\newcommand{\curveOp}{\mathfrak{R}}
\newcommand{\cAct}{\cdot}

\newcommand{\opclosure}[1]{{\overline{#1}}}

\newcommand{\Rem}[1]{}

\usepackage{imakeidx}
\makeindex
\newcommand{\Indexed}[1]{#1\index{#1}} 

\begin{document}
\title{\textsc{Quantum diffusion  on almost commutative spectral triples and spinor bundles 
}}
\author{Sita Gakkhar\footnote{\href{mailto:sgakkhar@caltech.edu}{email}.\newline
\emph{Acknowledgements:} The author would like to thank Matilde Marcolli for useful discussions, comments and for supporting this and other projects. The author would also like to thank Juan Pablo Vigneaux for many thoughtful  discussions.}}
\date{December 27, 2022}
\maketitle

\begin{abstract}
    Based on the observation that \namecite{cacic_reconstruction}'s characterization of almost commutative spectral triples as Clifford module bundles can be pushed to endomorphim algebras of Dirac bundles, with the geometric Dirac operator related to the Dirac operator of the spectral triple by a perturbation, the question of complete positivity of the heat semigroups generated by connection laplacian and Dirac and Kostant's cubic Dirac laplacians is approached using spin geometry and \CStar-Dirichlet forms. The geometric heat semigroups for on endomorphosm algebras of spinor bundles are shown to be quantum dynamical semigroups and the existence of covariant quantum stochastic flows associated to the heat semigroups on spinor bundles over reductive homogeneous spaces is established using the construction of \namecite{sinha_qsp}. 
\end{abstract}


\section{Introduction}

Almost commutative spectral triples are a class of noncommutative geometries defined to develop the noncommutative geometric approaches to the standard model of particle physics (see, for instance, \cite{suijlekom}). \namecite{cacic_reconstruction} characterizes such spectral triples as endomorphism algebras of spinor bundles. In this short article the heat semigroup associated to almost commutative spectral triples and spinor bundles is studied. The semigroup is shown to be a quantum dynamical semigroup. When the underlying manifold is a reductive homogeneous space, using the quantum stochastic calculus of \namecite{sinha_qsp}, the existence of a quantum stochastic dilation of Evans-Hudson type is established. The construction from \cite{sinha_qsp} extends \cite{goswami2000stochastic} to semigroups with unbounded generators, and while \cite{belton_unbounded_quantum_flow} is an alternative, the covariant construction is more natural on homogeneous spaces. This question of existence of such dilation -- which can be viewed as diffusion -- on the spectral triple has relevance in light of recent results of \cite{connes_spectral_entropy,khalkhali_second_quant} relating entropy for second quantization on the fermionic Fock space and the spectral action for the spectral triple. The appearance of Brownian bridge integrals expansion of the spectral action in \cite{farzad_brownian_bridge} is also suggestive of a deeper connection between noncommutative geometry and probability as Wiener space and boson Fock spaces are isomorphic. Also, of relevance is the stochastic quantization considered by \cite{albeverio_stochastic_quantization} where related ideas on Grassman algebras are explored.  
 
\subsection{Organization}

Section \ref{sec_qflows} introduces the quantum dynamical smigroups, dilations and relevant background; section \ref{sec_almost_com} introduces {\'C}a{\'c}i{\'c}'s results and gives the characterization in terms of spinor bundles. Section \ref{sec_markov_spinor} considers complete positivity on product almost commutative spectral triples and   considers more general almost commutative spectral triples (\S~\ref{sec_hilbert_schmidt_cp}). Section \ref{section_evans_hudson} introduces framework from \namecite{sinha_qsp} and establishes the existence of Evans-Hudson dilation on reductive homogeneous spaces. The results are obtained by first showing that the $\dlaplace$ defines a \CStar-Dirichlet form and generates a completely Markov semigroup, and then over reductive homogeneous spaces the covariance and regularity conditions required by \cite{sinha_qsp} are satisfied. 

Some notational conventions: for Hilbert space, $\Hs$, $\Bc(\Hs)$ will denote bounded operators on $\Hs$. $\I$ will denote $\sqrt{-1}$. For $m,n\in \NN$, $[n], [m:n]$ will denote $\{1,2\dots n\}$ and $\{m,m+1\dots n\}$ respectively. $\Lie{G}$ will denote the Lie algebra of Lie group $G$. $(M, g)$ will denote a smooth manifold $M$ with a Riemannian metric. After fixing a local orthonormal frame, $e_j$, the connection $\grad_j$ will be used interchangeably for $\grad_{e_j}$. Following \cite{lawson_spin}, by a Riemannian connection we mean a metric connection not necessarily torsion free; the canonical Riemannian connection is taken as the torsion free Riemannian connection. Throughout we will restrict to compact spin manifolds since that is the object in the reconstruction theorems from noncommutative geometry. Additionally, we will work with even dimensional manifolds with empty boundary\footnote{\label{rem_why_oeven_dim} The empty boundary requirement is needed as the Dirac operator may not be symmetric otherwise (see, for instance, \cite[eq~II.5.7]{lawson_spin}). The even dimensionality is used in two places: to use the characterization of Clifford module bundles as twisted spinor bundles, and the triviality of the center of even dimensional Clifford algebras over fields of characteristic not equal to 2\cite[\S~2.2]{meinrenken_clifford}. Since the odd dimensional Clifford algebras decompose as direct sums of two copies of even dimensional Clifford algebras, the difficulty is not fundamental.}. $\ClifA(E, q)$ will denote the Clifford algebra (bundle) over the vector space (vector bundle) $E$ with quadratic form $q$. For Riemannian manifold $(M, g)$, $\ClifA(M)=\djunion_{m\in M}\ClifA((T^*_mM), -g)$. $\dlaplace$ is used for geometric Dirac laplacian, i.e. associated to the given connection, and $D^2$ when $D$ may not be the geometric Dirac operator.

\section{Quantum stochastic flows}\label{sec_qflows}
 
\subsection{Quantum dynamical semigroups}
 
Suppose $\Ac$ is a $*$-algebra $\Ac$, then  $\mat_n(\Ac)\isomorph \Ac\tsr \mat_n(\CC)$. Recall the notion of complete positivity --
 
\begin{definition}For unital $*$-algebras, $\Ac_1, \Ac_2$ \begin{enumerate}
    \item A linear map $T:\Ac_1\to \Ac_2$ is positive if $T((\Ac_1)_+)\subset (\Ac_2)_+$.
    \item  $T$ is completely positive if for all $n\in \NN$, $T_n:=T\tsr 1_n: \mat_n(\Ac_1)\isomorph\Ac_1\tsr \mat_n(\CC) \to \Ac_2\tsr \mat_n(\CC)\isomorph \mat_n(\Ac_2)$, $T_n([a_{ij}]) = [T(a_{ij})]$, is positive. This is equivalent to $\sum_{i,j\in[n]}b^*_i \phi(a_i, a_j)b_j \geq 0$ with $n\in \NN, a_i\in \Ac_1, b_i\in \Ac_2, i\in[n]$
\end{enumerate}
\end{definition}
 
\begin{definition}[Quantum Dynamical Semigroups] A quantum dynamical semigroup
(Q.D.S.) on a \CStar-algebra $\Ac$
is a strongly continuous, contractive semigroup $T_t$ such that each $T_t:\Ac\to \Ac$ is a completely positive map. A semigroup is conservative if for all $t$, $T_t(1) = 1$ (equivalently $\Lc(1)=0$ for the generator $\Lc$ of $T_t$). The semigroup is of class $C_0$ if $\lim_{t\to t_0} T_tx = T_{t_0}x$ for all $x, t_0$.
\end{definition}

\begin{definition}(Covariant quantum dynamical semigroups) Let $G$ be a locally compact group acting on \CStar-algebra by $\alpha:G\to \Aut(\Ac)$ with $\alpha_g$ denoting $\alpha(g)$. A quantum dynamical semigroup $(T_t)$ is covariant with respect to $G$ if for all $t\geq 0, g\in G$, $T_t\circ \alpha_g=\alpha_g\circ T_t$, equivalently $\Lc\circ \alpha_g = \alpha_g\circ \Lc$ where $\Lc$ generates $(T_t)$.
\end{definition}
 
 
 
 

\begin{definition}A \Indexed{conditional expectation} is a linear map, $\EE:\Nc\to \Mc$, between $*$-algebras $\Mc, \Nc$, satisfying $\Mc\subset \Nc, \EE[1]=1$ and for any $M_i\in \Mc, N\in \Nc$ $\EE[M_1NM_2] = M_1\EE[N]M_2$ .
\end{definition}
 
\begin{definition}[Stochastic dilation]
For quantum dynamical semigroup $(T_t), t\geq 0$ on a \CStar~ (or von Neumann algebra) $\Mc$, a quantum stochastic dilation is a family of $*$-homomorphisms, $j_t:\Mc\to \Nc$, where $\Nc$ is a $*$-algebra
 with conditional expectation $\EE_0:\Nc\to \Mc$ satisfying $T_t = \EE_0[j_t]$.
\end{definition}
 
\noindent A stochastic dilation on the Fock space will be considered as a quantum stochastic dilation.

\begin{definition} For a \CStar or von Neumann algebra, $\Ac$, $\Ac''$ will denote the bicommutant. $\Ac''$ is a von Neumann algebra. For a Hilbert space $\Hs$, the free Fock space, $\Fock^f(\Hs)$ is the sum of the boson (symmetric) and fermion (antisymmetric) Fock spaces, $\Fock^s(\Hs), \Fock^a(\Hs)$,  $\Fock^f(\Hs) = \Fock^a(\Hs)\oplus \Gamma^s(\Hs)$. The symmetrization operator defines the map from free to boson Fock space, $\Fock^f(\Hs)\to \Fock^s(\Hs)$ by $\Symm(\tensor_{i\in [n]}g_i) = 1/(n-1)! \sum_{\sigma\in S_n} \tensor_{i\in[n]}g_{\sigma(i)}$. For a subspace $\Vs\subset\Hs$, $\ExpVec{\Vs}\subset \Fock^s(\Hs)$ denotes the $\CC$-linear span of exponential vectors $$\ExpVec{v} = \oplus_{k\in \NN} v^k/\sqrt{k!}, v\in \Vs$$
There's an inner product on the $\Fock^s(\Hs)$ induced by the inner product on $\Hs$, $\ip{\ExpVec{u}, \ExpVec{v}} = \exp \ip{u, v}$.
The exponential vectors are linearly independent and total in the boson Fock space. We will denote boson Fock space $\Fock^s(\Hs)$ by $\Fock(\Hs)$. 
\end{definition}
 
\subsection{Brownian motion as a stochastic dilation}
 
Viewing Brownian motion on $(M, g)$ as a diffusion generated by the Laplace-Beltrami operator, it's noted that the Feynman-Kac formula for a Riemannian manifold\cite[Thm~3.2]{riemannian_feynman_kac}, $(M, g)$, for the operator $H:=\tfrac{1}{2}\laplace_{\C(M)}+V, u\in \C^4(M), V\in \C(M)$, with Laplace-Beltrami operator, ${\laplace}_{\C(M)}$,  acting on the $\C^2(M)$ gives $
(e^{-tH}u)(x) = \int_{W(M)} \exp(\int_0^t V(\omega(s)) - 1/6\cdot\kappa_M(\omega(s)) ds)  u(x)\D W^x_M(d\omega)/N(u, \kappa_M, \D W^x_M(d\omega))
$ where $\D W^x_M(d\omega)$ denotes the Wiener measure on $\C(M)$, $u\in \C^4(M)$ and $\kappa_M$ is the scalar curvature of $M$ and $N(u, \kappa_M, \D W^x_M(d\omega))$ a normalization depending on $\kappa_M, u$ and $\D W^x_M(d\omega)$. This can be thought of as a stochastic dilation of heat semigroup on $\C^\infty(M)$ to the Wiener space, $W(M)$, on $M$, the integral with respect to the Wiener measure playing the role of the conditional expectation.
 
The Ito-Wiener-Segal isomorphism \cite{nualart2006malliavin} 
provides the bridge to quantum probability:
for a separable Hilbert space $\Hs$, a stochastic process $\Wc=\{W(h), h\in \Hs\}$ defined on a complete probability space $(\Omega, \Fc, \PP)$, with each $W(x)\in \Wc$ a centered Gaussian satisfying $\EE(W(h), W(g)) =\ip{h, g}_\Hs$, is called an isonormal Gaussian process, and for the $\sigma$-field $\Gc$  generated by $w\in\Wc$ for an appropriate isonormal Gaussian process $\Wc$ (see, for instance, \cite[\S~1.1]{nualart2006malliavin}), $L^2(\Omega, \Gc, \PP)$ is isomorphic to the boson Fock space $\Fock(\Hs)$. Additionally, when $\Hs$ is the space $L^2(T, \Bc, \mu)$ where $\mu$ is $\sigma$-finite without atoms over a measure space $(T, \Bc)$, $W(h)$ can be regarded as stochastic integrals, with polynomials in $W(h)$ dense in $L^2(\Omega, \Gc, \mu)$. The canonical example\cite[Ex~19.9]{parthasarathy_qsp} is for $\Hs:=L^2(\RR_{\geq 0})$ where $\Fock(L^2(\RR_{\geq 0}))\isomorph L^2(\C(\RR_{\geq 0}), \PP_{\txt{Wiener}})$.
Through the Ito-Wiener-Segal isomorphism between Wiener space $W(M)$ and the associated Fock space, the  heat semigroup has a stochastic dilation on the Fock space. This dilation corresponds to a flow for a Evans-Husdon type quantum stochastic differential equation (qsde) introduced next. A  process satisfying a qsde of this type is considered as a quantum diffusion process.

\subsection{Quantum stochastic dilation of Evans-Hudson type}
 
On a smooth manifold, $M$, a (homogeneous) flow is a smooth map $\phi:\Rnonneg\times M\to M$, $\phi_t(m):=\phi(t, m)$, satisfying $\phi(t+s, m)=\phi(s, \phi(t, m))$, $\phi(0, m)=m$. The flow induces a 1-parameter semigroup, $(j_t)_{t\geq 0}:\C^\infty(M)\to \C^\infty(M)$, $j_t(f):=f\circ \phi_t^{-1}$ with the infinitesmal generator $\Lc$ following the differential equation\cite{Hudson_quantum_stochastic_flows}, \begin{equation}\label{eq_flow}
\dfrac{d }{dt}j_t(f) = j_t(\Lc(f)), ~~\txt{with}~ j_0(f)=f,~ \Lc(f)=\dfrac{d}{dt}\bigg|_{t=0}j_t(f),~f\in \C^\infty(M)
\end{equation}
 
The classical stochastic flow can be viewed as a stochastic process $\psi_t$ taking values in diffeomorphism group of $M$ which satisfies the flow property almost surely (see, for instance, \cite[Ch~3]{kunita_stochastic_flows}). Now solutions to stochastic differential equations (sde) on manifolds generate stochastic flows, the stochastic version of flow equation is obtained by introducing Wiener process terms into eq~\ref{eq_flow} yields
$$
\dfrac{d }{dt}j_t(f) = j_t(\Lc(f)) + \sum_{j\in [n]}j_t(b_j(f))dB_j
$$
for linear maps $b_k$, and components $B_j$ of $n$-dimensional Brownian motion $B$ on $M$ with sample space $\Omega$. Algebraically, $j_t$, are now $*$-algebra homomorphisms, $j_t:\Bc(M\times \Omega)\supset C^\infty(M)\to \Bc(M\times \Omega) $ for the space of bounded measurable functions, $\Bc(M\times \Omega)$, on $M\times \Omega$. Note that $C^\infty(M)$ is embedded in $\Bc(M\times \Omega)$ by trivially extending to $M\times \Omega$. In the integral form, the quantum analog of this sde can be defined on the Fock space.
 
For a finite dimensional Hilbert space $\Vs$, set $\Hs=L^2(\Rnonneg, \Vs):=L^2(\Rnonneg)\tsr\Vs$. $\Hs$ decomposes as $\Hs=\Hs_t\oplus \Hs^t$ where $\Hs_t=L^2([0, t))\tsr \Vs, \Hs^t=L^2([t, \infty))\tsr \Vs$. On the Fock space, $\Fock(\Hs)=\Fock(\Hs_t)\algtsr \Fock(\Hs^t)$. Given an ``initial'' Hilbert space $\Hs_0$, set $$
\tilde \Hs_t = \Hs_0\tsr\Fock(\Hs_t), \tilde \Hs^t = \Hs_0\tsr\Fock(\Hs^t), \tilde \Hs = \Hs_0\tsr\Fock(\Hs)
$$ then for a class of $\Rnonneg$-indexed operator families on $\tilde \Hs$, $\Lambda^i_j$, $i, j\in[0:\dim \Vs_0]$, called the fundamental processes (or quantum noises, which corresponds to the annihilation, creation and conservation processes on the Fock space), the quantum stochastic integral $\int_0^t \sum_{i,j}E^j_i d\Lambda^i_j$
can be defined for processses $(E^j_i)_{t\in \RR_{\geq 0}}$ that are regular (i.e. the map $t\to (E^j_i)_t(u_0\tsr \ExpVec{u})$ is coninuous with a growth condition on $\lVert(E^j_i)_t(u_0\tsr \ExpVec{u})\rVert$) and each $(E^j_i)_t$ is adapted where a process $X_t:\tilde{\Hs}\to  \tilde{\Hs}$ is adapted if there exists $Y_t:\Hs_0\tsr \ExpVec{\Hs_t}\to \Hs_0\tsr \Fock(\Hs_t)$, so that $X_t=Y_t\tsr 1_{\Fock(\Hs^t)}$, i.e. $X_t$ does not look into the future -- same as the classical notion of adaptedness. For brevity, we do not detail the construction, but refer to standard references \cite{parthasarathy_qsp, sinha_qsp, 
attal2006open}\footnote{These constructions can be done in noncommutative probability in general, the theme is of replacing the $\sigma$-algebra of events by the possibly noncommutative algebra of adapted stochastic processes (viewed as $\RR$-indexed random variables), the decomposition $\Hs=\Hs_t\oplus \Hs^t$ reflecting the past and future $\sigma$-algebras for the filtration.}  
 
The stochastic calculus can be developed on operator algebras in similar manner to Hilbert spaces\cite[Ch~5]{sinha_qsp} and the stochastic flow can be defined by extending the classical picture: for a dense $*$-algebra $\Ac_0\subset\Ac$ with $\Ac\subset \Bc(\Hs_0)$ unital, the quantum stochastic flow $(j_t)_{t\geq 0}$ is family
of injective $*$-homomorphism, $j_t:\Ac_0\to \Bc(\Hs)$, such that for all $a\in \Ac$, each $j_t(a)$ is an adapted process and there exists $\{\lambda^i_j:i,j\in [0:\dim \Vs]\}$, called the structure maps, with $$
j_t(a) = a\tsr \one + \int_0^t
\sum_{i,j}j_t(\lambda^i_j(a))d\Lambda^j_i$$
 
Equivalently, in differential form, $dj_t(a)=\sum_{i,j}j_t(\lambda^i_j(a))d\Lambda^j_i$ with $j
_0=\one$
. Flows of this form, with $j_t$ satisfying some additional constraints, are called Evans-Hudson flows\cite[\S~27,28]{parthasarathy_qsp}\footnote{From \cite{parthasarathy_diffusions} note, any Markov chain on countable state space can be realized as Evans-Hudson flow.}. In particular, Brownian motion on $\RR$ can be realized as Evans-Hudson dilation on the Fock space $\Gamma(L^2(\Rnonneg))$ defined below
by specializing to $\Ac=L^\infty(\RR)$ viewed as operators on $\Hs=L^2(\RR)$, $\Vs$ fixed as trivial and using the Fock space-Wiener space dictionary provided by the Wiener-Ito-Segal isomorphism.

\begin{definition}(Evans-Hudson dilation\cite[Def~6.0.2]{sinha_qsp}) For a quantum dynamical semigroup $(T_t)_{t\geq 0}$ with generator $\Lc$ on \CStar-algebra $\Ac\subset \Bc(\Hs)$, a family of $*$-homomorphisms, $(j_t)_{t\geq 0}:\Ac\to \Ac''\tsr \Bc(\Fock(L^2(\Rnonneg)\tsr \Vs))$  satisfying -- \begin{itemize}
\item There exist maps $J_t:\Ac\algtsr  \ExpVec{L^2(\Rnonneg)\tsr \Vs}\to \Ac''\tsr \Bc(\Fock(L^2(\Rnonneg)\tsr \Vs))$, $J_t(a\tsr e(f))u:=j_t(a)(ue(f))$ such that for an ultra-weakly dense subalgebra $\Ac_0\subset \Ac$, $\Dom(\Lc)\subset \Ac_0$, on $\Ac_0\tsr {L^2(\Rnonneg)\tsr \Vs}$ the Evans-Hudson flow qsde $$
dJ_t = J_t(a_\delta(dt) + a_\delta^\dagger(dt) + \Lambda_\sigma(dt) + \one_\Lc(dt)), ~J_0 =\one
$$
holds, where $a_\delta, a_\delta^\dagger,  \Lambda_\sigma, \one_L$ are structure maps as defined \cite[\S~5.4]{sinha_qsp}; $J_t$ as a quantum stochastic process is regular and adapted.
 
\item $j_t$ is a dilation of $T_t$ in the following sense:  for all $u, v\in \Hs, a\in \Ac$, $\ip{v\ExpVec{0}, j_t(a)u\ExpVec{0}}=\ip{v, T_t(a)u}$
\end{itemize}
\end{definition}

\section{Almost commutative spectral triples as spinor bundles}\label{sec_almost_com}

A spectral triple is three basic pieces of data, $(\Ac, \Hs, D)$, where $D$ is symmetric operator on the Hilbert space $\Hs$, and a *-algebra of bounded operators on $\Hs$, $\Ac\subset \Bc(\Hs)$. The operator $D$ is allowed to self-adjoint and unbounded but with $[D_i, a]$ bounded for all $a\in \Ac$. A compact Riemannian spin manifold $(M, g)$ can be characterized by the canonical spectral triple, $\mathfrak{A}_M := (\C^\infty(M) , L^2(\Sc), D_M; J_M , \gamma_M)$ where $\Sc$ is the spinor bundle, $\C^\infty(M)$ is the $*$-algebra of smooth functions interpret as operators acting on $L^2(\Sc)$ by multiplication and $D_M$ is the Dirac operator associated with the Levi-Civita connection on the spinor bundle, and the data of a spectral triple has been supplemented with a $\ZZ_2$ grading operator $\gamma_M$ on $\Hs$ and an anti-unitary operator $J:\Hs\to \Hs$, called the real structure, which makes $\Hs$ an $\Ac-\Ac$ bimodule from a left $\Ac$-module. Such spectral triples can be characterized abstractly; Connes reconstruction theorem recovers the Riemannian spin structure from the abstract spectral triples\cite[Thm~11.2]{gracia_ncg}.
 
A finite noncommutative space is the finite spectral triple, $\mathfrak{A}_F:= (\Ac_F , H_F , D_F)$, with $\dim H_F$ finite. This is supplemented with a real structure and a grading, $(J_F , \gamma_F)$. A product almost commutative spectral triple is the globally trivial bundle, $$
M\times F := (\C^\infty(M)\tsr A_F,L^2(M, S\tsr H_F
), D_M\tsr 1 + \gamma_M\tsr D_F; J_M\tsr J_F , \gamma_M \tsr \gamma_F)$$
 
\namecite{cacic_reconstruction} expands the definition of product almost commutative spectral triples include
include non-trivial algebra bundles over the base space. This is formalized without appeal to the explicit product structure as an abstract almost commutative spectral triple:\begin{definition}\label{def_cacic_abstract}(\cite[Def~2.16]{cacic_reconstruction})
A spectral triple $(\Ac, \Hs, D)$, $\Bc\subset\Ac$ a central, unital $*$-subalgebra is an abstract almost-commutative spectral triple over the base $\Bc$ if $(\Bc, \Hs, D)$ is a commutative spectral triple of Dirac type\cite{cacic_reconstruction}, and for all $a \in \Ac, [D, a]^2\in \Ac$, additionally \begin{enumerate}
    \item For all $a\in \Ac, b\in \Bc$, $[[D, b], a] = 0$.
    \item $\Ac$ is an even finitely generated projective $\Bc$-module and a $*$-subalgebra of the algebra $\End_{\Bc+i\Bc}(\Hs_\infty)$ where $\Hs_\infty = \cap_{k\in \NN}\Dom D^k$
\end{enumerate}
\end{definition}
 
The concrete realization of the abstract almost-commutative spectral triple is constructed by appeal to Connes's reconstruction theorem\cite[Ch~11]{gracia_ncg}, and the following global analytic equivalent formulation is obtained, and this is formulation that we work with.
 
\begin{definition}\cite[Def~2.3]{cacic_reconstruction} An almost-commutative spectral triple is a spectral triple of the
form $$(\C^\infty(X, A), L^2(X, H), D_0)$$ for a compact oriented Riemannian manifold $X$, $H$ a self-adjoint Clifford module bundle, $A$ a real unital $*$-algebra subbundle of $\End^+_{\ClifB(X)}(H)$, and $D_0$ is a symmetric Dirac-type operator on $H$, where $\End^+_{\ClifB(X)}(H)$ are the even endomorphisms of $H$ that supercommute with the Clifford action $c:T^*X\to \Endo{H}$ defined by $D$.
\end{definition}
 
\begin{remark}\label{rem_cliffod_action_commutation}
Recalling that a $\ZZ_2$ graded $\KK$-algebra, $A=A^0\oplus A^1$, with $A^i\cdot A^j\subset A^{i+j}$, the supercommutator $[\cdot, \cdot]_s$ is the map $[a^i, b^j]_s = a^ib^j-(-1)^{ij}b^ja^i$ for $a^i\in A^i, b^j\in B^j$. As the Clifford action, $c:T^*X\to \Endo{H}$, and as $\End^+_{\ClifA(X)}(H)$ consists of even endomorphisms, $\phi\in \End^+_{\ClifA(X)}(H)$ commutes with $c$, $\phi\circ c=c\circ \phi$.
\end{remark}

 
\subsection{Structure of Dirac bundles}
 
Recall the Clifford algebra $\ClifA(V, Q)$ is the algebra generated over the vector space $V$ by the relation 
$v^2=Q(v)1$ where $Q$ is a quadratic form on $V$. It satisfies the following universal property: any linear map $f:V\to \Ac$, $V$ a vector space, $\Ac$ a unital associative $\KK$-algebra, with $f(v)\cdot f(v) = Q(v)1$ extends uniquely to a $\KK$-algebra homomorphism $\tilde f:\ClifA(V, Q)\to \Ac$.  $\ClifA(V, Q)$ comes with a $\ZZ_2$ grading, $\chi(v_1\cdots v_k) = (-1)^kv_1\cdots v_k$, that yields the decomposition, $\ClifA(V, Q)= \ClifA(V, Q)^0\oplus \ClifA(V, Q)^1$. Specializing to $\RR^n$, fix $Q_n=\sum_{n}x_i^2$, define $\ClifA^+_n=\ClifA(\RR^n, Q_n), \ClifA^-_n=\ClifA(\RR^n, -Q_n)$ and $\ClifA^n=\ClifA(\CC^n, Q_n)$ which is $\ClifA_n^+\tsr_\RR \CC, \ClifA_n^-\tsr_\RR \CC$. The grading comes from the chirality operator $\gamma_{n+1}$ on $\ClifA_n$ is given by $(-\I)^me_1\cdots e_n$ where $e_i$'s generate $\ClifA_n$ and $n=2m$ if even and $n=2m+1$ for odd. This can be carried over to a vector bundle --
 
\begin{definition} A Clifford structure on a vector bundle $E$, is a bundle morphism $c:T^*M \to \End(E)$, $\{c(u), c(v)\}=-2g(u, v)1$. $c(v)\in \End(E)$ denotes the ``Clifford multiplication by $v$'', and the pair $(E, c)$ is the \Indexed{Clifford bundle}. The Clifford bundle $E\to M$ is $\ZZ_2$ graded if there's a decomposition $E=E^+\oplus E^-$ such that $c(\alpha)$ for each $\alpha\in T^*M$ is an odd endomorphism: $c(\alpha)(\MSect(E^\pm)) = \MSect(E^\mp)$. A vector bundle with a Clifford structure is a Clifford module bundle.
\end{definition}

\begin{remark}
Note that $c:\Omega^1(M)\to \MSect(\Endo{E})$. Actually, $c$ is the action of the full Clifford algebra, $c:\ClifB(M)\to \MSect(\Endo{E})$, but because $c:T^*M\to \MSect(\Endo{E})$ satisfies $\{c(u), c(v)\}=-2g(u, v)$, the full action follows by using the universal property of Clifford algebras; $(c, E) $ is a representation of $\ClifB(M)$.\end{remark}
 
\noindent On any Riemannian manifold $(M, g)$, there exists a canonical Clifford bundle, $\ClifB(T^*M, -g):=\ClifB(M)$. A Clifford module bundle is any bundle that carries an action of the Clifford bundle.
 
A Dirac bundle $S$ over a $(M, g)$ is a Clifford module bundle with a connection $\grad^S$ that is compatible with the Clifford  multiplication --
\begin{itemize}
\item for all $\sigma_i\in S_x, e\in T_xM, \norm{e}=1$, $e$ acting on $\sigma_i$ by clifford multiplication, $\ip{e\cdot\sigma_1, e\cdot\sigma_2} = \ip{\sigma_1, \sigma_2}$ (as $e^2=-1$, this yields the skew-hermiticity, $\ip{e\cdot\sigma_1, \sigma_2} = -\ip{\sigma_1, e\cdot\sigma_2}$)
\item  $\grad^S(\phi\cdot \sigma) = (\grad^{\ClifB(M)}\phi)\cdot \sigma + \phi\cdot \grad^S\sigma$
\end{itemize}

\noindent For clarity it is useful to separate out the algebraic Clifford structure from the geometric piece.
 
\begin{definition}A \Indexed{Dirac bundle}, $(E, c, h, \grad, M, g)$, is a Clifford bundle $(E, c)$ over $(M, g)$ with a hermitian metric $h$ on $E$ and Clifford connection, $\grad$, compatible with $h$ such that for all $\alpha\in \Omega^1(M)$ the following holds: \begin{itemize}
\item $c(\alpha)\in \Endo{E}$ is skew-Hermitian
\item For $X\in\MSect(TM), u\in \SmoothSect{E}$, $\grad^M$ the Levi-Civita connection on $M$, $\grad_X(c(\alpha)(u)) = c(\grad_X^M\alpha)u + c(\alpha)(\grad_Xu)$
\end{itemize}
The Dirac structure is the tuple $(\grad, h)$ associated to $(E, c)$.
\end{definition}
 
\begin{definition}[Geometric Dirac operator] A \Indexed{geometric Dirac operator}
is a Dirac operator, $\diracop$, that is associated to a $(E, c, h, \grad)$ Dirac structure over $(M, g)$ by $$
\diracop := c\circ \grad: \MSect{E}\map{\grad}\MSect{(T^*M\tsr E)}\map{c}\MSect{E}
$$
\end{definition}
 
In local coordinates, after fixing a basis $(e^i)$ of $T^*M$ and the corresponding dual basis $(e_i)$, $\grad s\in \MSect(S\tsr T^*M)$ can be expanded in this basis as $\sum_i e^i \tsr \grad_{e_i}s$. Composed with the clifford action this gives that the geometric Dirac operator acts by $\MSect(S)\ni \sigma \to \sum_i  e^i\cdot\grad_{e_i}\sigma\in \MSect(S)$. More generally Dirac operator can be defined as a first order partial differential operator on the sections of any left $\ClifB(M)$ module bundle. 
 
\begin{definition}(Generalized laplacian and Dirac operator)\begin{itemize}
    \item A \Indexed{generalized laplacian} $\laplace$ is a second order differential operator on a vector bundle $E$ with $\sigma_2(L)(x, \xi)=|\xi|^2$.
    \item A Dirac-type operator on a Clifford module bundle $E$ with Clifford action $c$ over $(M, g)$ is a first order differential operator $D$ such that $[D, f] = c(df)$ for all $f\in \C^\infty(M)$. 
\end{itemize}
\end{definition}
 
Every  Dirac operator $D$ on the vector bundle $E$ over $M$, induces a Clifford action of $T^*M$ on $E$ by $c(df):=[D, f]$ for $f\in \C^\infty(M)$, and conversely, associated to any Clifford action $c$, the operator satisfying $[D, f]=c(df)$ is a Dirac operator (see, for instance, \cite[Prop~3.38]{berline_heat}).
 
\begin{definition}(Spinor bundle) For any oriented vector space $V$, $\dim V=2k$, the spinor module is the unique $\ZZ_2$ graded Clifford module $S=S^+\oplus S^-$ with $\ClifA(V)\tsr \CC=\End(S)$.
The spinor bundle $\Sc$ over a $2k$-dimensional manifold $M$ with a spin structure is the associated bundle $\Spin(M)\times_{\Spin(n)} S$.
\end{definition}

Every $\ZZ_2$ graded complex $\ClifA(V)$-module $E$ is isomorphic to $W\tsr S$, and given $E$, $W$ can be recovered by $W=\Hom_{\ClifA(V)}(S, E)$ with trivial $\ClifA(V)$ action, that is, the Clifford action on $E$ is the Clifford action on the $S$ component, $e\cdot(w\tsr s):=w\tsr(e\cdot s)$, and $\End(W)\isomorph \End_{\ClifA(V)}(E)$. The fibers, $\Sc_x$,  are isomorphic to $S$, and, therefore, over local a trivialization , $(U, \phi_U)$, the statements about $S$ carry  over to $\Sc\big|_U$ and to the bundle $\Sc$: every Clifford module $H$ over $M$ is a twisted spinor bundle, $H=\Wc\tsr \Sc$ with $\Wc\isomorph \End_{\ClifB(X)}(\Sc, H)$, $\End(\Wc) \isomorph \End_{\ClifA(X)}(H)$ (see, for instance, \cite[Prop~3.35]{berline_heat}).

This yields that on spin manifolds, associated to Clifford structures, Dirac structures exist\footnote{see, for instance, \cites[Prop~11.1.65]{nicolaescu_geometry}[Cor~3.41]{berline_heat}} -- locally $H=\Wc\tsr \Sc$, the tensor product connection of the Levi-Civita connection on $\Sc$ and any connection on $\Wc$ is compatible with  the Clifford action, and then the global version follows by a partition of unity argument. 
 
Suppose $(\C^\infty(X, A), L^2(X, H), D_0)$ is an almost commutative spectral triple with generalized Dirac operator $D_0$, $H$ a spinor bundle over compact spin manifold $M$. \cite[Thm~2.17]{cacic_reconstruction} gives a metric $Q$ on $H$ which corresponds to the Clifford action associated to $D_0$ on $H$. By above, there exists a Dirac structure on $H$ arising from the Clifford action induced by $D_0$. $H$ being a Clifford module bundle is a twisted spinor bundle $\Wc\tsr\Sc$ with twisting space $\Wc$; the connection on $\Sc$ is the spinor connection and on $\Wc$ it can be chosen as the canonical Riemannian connection for the metric induced from $Q$. If, however, the bundle $H$ comes with a Dirac structure, then choice to use canonical Riemannian connections on the twisting space and the spinor bundle is not necessary and the given Dirac structure can be used. The canonical Riemannian connections are chosen the connection on $H$ is torsion free. However, often connections with non-vanishing torsions are of interest and arise naturally: for example, the canonical connection on a homogeneous space and connection associated to Kostant's cubic Dirac operator may not be torsion free\cite{agricola2003connections}; connections with totally anti-symmetric non-zero, torsion also have relevance to models of gravity.
 
\begin{observation}\label{obv_dirac_perturbation_form} Suppose $D$ is the geometric Dirac operator for the Dirac structure.  Then $D_0$ and $D$ give the same Clifford action, and, therefore, $D-D_0=A$ for some odd endomorphism, $A\in  \MSect(\End^-(H))$. When $\Wc$ is $\ZZ_2$-graded, and the $D_0, D$ are odd with respect to the grading and $\gamma D_0=-D_0\gamma, \gamma D=-D\gamma$, then $A$ is anticommutes with $\gamma$. Using this, it follows that $D-D_0=A$ is in fact associated with a connection potential: $H=\Wc\tsr \Sc$ implies $A\in  \MSect(\ClifA(X)\tsr\End(\Wc))$, but as vector spaces $\ClifA(X)$ and $\Lambda T^*X$  are isomorphic, so $A$ specifies a form $\End(\Wc)$-valued form $\Omega_A$ defined in local geodesic coordinates $(e_i)$ by $\Omega_A(e_i) = \gamma A$ for each $i$. Because $\gamma A=-A\gamma$,  $\Omega_A$ is anti-symmetric and, therefore, defines a metric compatible connection locally by $d + \Omega_A$. More formally, this is the statement that the Dirac operators for a Clifford action on a twisted spinor bundles are in one--one correspondence with (super)connection on the twisting space (see, for instance, \cite[Ch~3]{berline_heat}).
\end{observation}

Remark~\ref{rem_cliffod_action_commutation} leads to the following result on structure of almost commutative type endomorphism algebra $A\subset \End^+_{\ClifB(X)}(H)$, where $H$ is Clifford module bundle over a Riemannian manifold, $(X, g)$, therefore, a twisting of the complex spinor bundle $\Sc$, $H=\Wc\tsr \Sc$.  
 
\begin{thm}\label{thm_commutes_with_cliffod_action} If $\alpha\in A\subset \End_{\ClifB(X)}^+(H)$ then $\alpha=w_\alpha\tsr 1$ for $w_\alpha\in \End(\Wc)$ up to multiplication by $f\in \C(X)$. That is, as a module over $\C(X)$, $C(X, A)$ is generated by endomorphims of form $w_\alpha\tsr 1$.\end{thm}
 
\begin{proof} The proof is basically the observation that locally $\End(\Wc)\isomorph\End_{\ClifA(X)}(H)$ (see, for instance, \cite[Prop~3.27]{berline_heat}) (i.e. $A\isomorph W_A\subset \End(\Wc)$).
Now $\End(H)$ is the topolocal closure of  $\End(\Wc)\tsr\End(\Sc)$, where because $\Sc$ is the complex spinor bundle, $\End(\Sc)\isomorph\ClifA(X)\tsr \CC$.
 
Suppose $\alpha = \sum_i\alpha_{w,i}\tsr \alpha_{s,i} \in A
\subset \closure{\End(\Wc)\tsr \ClifA(X)\tsr \CC}$ where $\alpha_{s,i}\in \ClifA(X)\tsr \CC$.  Consider the Clifford action, $c:\ClifA(X)\to \End(H)$, $v\to c(v):=\sum_{i} w_i\tsr s_i \in \End(H)$ with $w_i\in \End(\Wc), s_i\in \ClifB(X)\tsr \CC$. By construction of the twisted spinor bundle, the Clifford action on the $\Wc$ piece is trivial so $w_i=1$ for all $i$, therefore, $c(v)=1\tsr v_s$ with $v_s\in \ClifA(X)\tsr \CC$.
 
Since $\alpha$ commutes with the Clifford action, $c(v)$, \begin{align*}
\sum_{i}\alpha_{w,i}\tsr v_s\alpha_{s,i} = (1\tsr v_s) \circ \sum_{i}\alpha_{w,i}\tsr \alpha_{s,i} &=  \sum_{i}\alpha_{w,i}\tsr \alpha_{s,i} \circ (1\tsr v_s) = \sum_{i}\alpha_{w,i}\tsr \alpha_{s,i}v_s
\end{align*}
 
In even dimensions, the canonical complex bundle $\Sc$ in the twisted spinor decomposition, $\Wc\tsr \Sc$, is irreducible Clifford module and $\ClifA(X)$ is a central simple algebra, therefore, $v_s$ runs over all elements in $\ClifA(X)\tsr \CC$
As $v_s$ is arbitrary, $\alpha_{s,i}$ lie in the center of $\ClifB(X)$. This can be seen locally -- choose a basis $(e_i)$ for $T^*X$, then the basis for $\ClifA(T^*M)$ is $(e_I)_{I\subset [\dim T^*X]}$. Expressing $\alpha$ in $e_I$'s gives $\sum_i \alpha_{w, i}\tsr \alpha_{s, i} = \sum_{I} k_I  \alpha_I\tsr e_I$ for $k_I$. Note that $e_i\cdot e_I = \pm e_I \cdot e_i$ for any $i$. Suppose $|I|> 0$. If $|I|$ is odd, then there exists $j\not\in I$, and $e_I \cdot e_j = -e_j \cdot e_I$ as it commutes past each $e_i$ for $i\in I$. If $|I|=2k$ with $e_I = e_{i_1}\cdot \dots e_{i_{2k}}$, then $e_I\cdot e_{i_{2k}} = -e_{i_1}\dots e_{i_{2k-1}}$, while $e_{i_{2k}}e_I = e_{i_1}\dots e_{i_{2k-1}}$ because there are $2k-1$ sign changes on moving across and then a final sign change from $e_{i_{2k}}^2=-1$. Therefore, $|I|=0$ for $e_I$ to commute with each $e_i$ but then $e_I\in \Zc(\ClifA(T^*X))$. 
Note there's no ambiguity in the crossnorm with respect to which $\End(\Wc)\tsr\ClifB(X)\tsr\CC$ is completed: the norm is the operator norm on $H$. The conclusion holds on the algebraic tensor product, and also the topological completion.
\end{proof}
 
This is consistent with the case for commutative spectral triples where the algebra $\C^\infty(M)$ acts by multiplication on the spinor bundle $L^2(S)$ and commutes with the Clifford action. Since the spinor endomorphism part is restricted to be trivial, if the twisting space is chosen as a trivial matrix bundle with Hilbert-Schmidt inner product, then a rough analogy between almost commutative and fuzzy spectral triples introduced in \cite{barrett2015matrix} becomes clear.

\subsection{Complete Markovity on product almost commutative spectral triples}\label{almost_commutative_spectral_triples}
 
For real even spectral triples, $(\Ac_i, \Hs_i, D_i; J_i, \gamma_i)$, $i\in\{1, 2\}$, that is, the spectral triples comes with a real structure $J_i$ and a grading $\gamma_i$ such that for all $a\in \Ac_i, \gamma_i a=a\gamma_i, \gamma D_i=-D_i\gamma$, the product is defined by $\Ac:=\Ac_1\tsr \Ac_2, \Hs=\Hs_1\tsr\Hs_2, D:=D_1\tsr 1+\gamma_1\tsr D_2, \gamma=\gamma_1\tsr\gamma_2, J=J_1\tsr J_2$. If the second triple is not even then resulting structure does not have a grading and the adjective even is dropped. Since the first triple is even and $D_1, \gamma_1$ anti-commute, $D^2=D_1^2\tsr 1+ 1\tsr D_2^2$. Note that $\C^\infty(M, A_F) = \C^\infty(M)\tsr A_F$ and the tensor products are $\ZZ_2$-graded
.
 
\begin{remark}\label{rem_z2_graded_norm_ambiguity}
While noncommutative geometry does not require that $\Ac_i$ be closed, for example, the requirement $[D_i, a]$ is bounded is only needed for $a\in \Ac_i$; however, the questions about quantum dynamical semigroups often presume norm closure. The algebras $\Ac_1,\Ac_2$ are only pre-\CStar-algebras, but can be completed in the respective $\CStar$-norm; for the canonical spectral triple, $\C^\infty(M)$, will have $\C(M)$ as the closure. 
Note that in the $\ZZ_2$-graded tensor product category, as with ungraded tensor prodcts, the commutative \CStar-algebras are characterized  nuclear\cite{crismale_z2_cstar_tensor}, so there's no ambiguity in the norm to use.  
\end{remark}
 
The product almost commutative spectral triple is the product of the canonical spectral triple of a Riemannian spin manifold, $\mathfrak{A}_M := (\C^\infty(M) , L^2(S), D_M; J_M , \gamma_M)$, and a finite noncommutative space, $\mathfrak{A}_F:= (A_F , H_F , D_F; J_F , \gamma_F)$, $$
M\times F := (\C^\infty(M)\tsr A_F,L^2(M, S\tensor H_F)), D_M\tsr 1 + \gamma_M\tsr D_F; J_M\tsr J_F , \gamma_M \tsr \gamma_F)$$
 
\noindent The following technical lemma will be useful: topological questions in tensor products of \CStar-algebras can be delicate; we will work with identifications as described in the statement, the point of this lemma is that such identifications behave well.
 
\begin{lemma}\label{lemma_cstar_embedding_by_unital_extension} Suppose $\Ac_1, \Ac_2$ are unital \CStar-algebras. Suppose at least one of $\Ac_1, \Ac_2$ is nuclear, so there's a unique crossnorm on $\Ac_1\tsr\Ac_2$, then the map, $\phi:\Ac_1\to\Ac_1\tsr\Ac_2, a\to 1\tsr a$, is a completely positive, homeomorphism onto its image. \end{lemma}
 
\begin{proof} The kernel of $\phi$ is trivial, and $\phi$ is positive as $a$ positive in $\Ac_1$ means $1\tsr a$ is positive in $\Ac_1\tsr\Ac_2$. Trivially $\phi$ is unital. From the $\RR$-linearity of the tensor product, it follows the map $\phi$ preserves norms. It also follows that $\phi\tsr \one_n$ also preserves norms, so $\phi$ is a unital, completely contractive map, hence is completely positive. Being contractive also implies continuity. The inverse map on the image, $\phi^{-1}$, $1\tsr a\to a$, is again unital and completely contractive: the same holds for $\phi^{-1}$ as well.   
\end{proof}
 
\noindent We note the following about the complete positivity of the Dirac heat semigroup for product almost commutative spectral triples:
 
\begin{theorem} The complete positivity of the semigroup $e^{-tD^2}$ is well defined.
\end{theorem}
\begin{proof}
By remark~\ref{rem_z2_graded_norm_ambiguity}, if both spectral triples are $\ZZ_2$-graded with $\ZZ_2$-graded tensor product, $\C(M)\tsr A_F$ , then as using the commutative algebra $\C(M)$ is nuclear in $\ZZ_2$-graded tensor product category, the cross norm is unique, and there's no ambiguity to the norm with which to complete the tensor product. 
 
The only other ambiguity to resolve is the order of the product: $\mathfrak{A}_F \times \mathfrak{A}_M$ versus $\mathfrak{A}_M \times \mathfrak{A}_F$. However, the complete positivity of the semigroup generated by $-D^2_{F\tsr M}:= -D_F^2\tsr 1 - 1\tsr D_M^2$ means the map $e^{-tD^2_{F\tsr M}}\tsr \one_n:\mat_n[A_F\tsr\C^\infty(M)]\to \mat_n[A_F\tsr\C^\infty(M)]$ is completely positive for each $t$. On the algebraic tensor product, the positivity of $a\tsr b$ and $b\tsr a$ is equivalent, so the oppositely ordered semigroup generated $e^{-tD_{M\tsr F}}\tsr \one_n:\mat_n[\C^\infty(M)\tsr A_F]\to \mat_n[\C^\infty(M)\tsr A_F]$ is also completely positive; the result also holds in the topological completion as the norm is independent of the order of the product.  \end{proof}

\begin{theorem} If $e^{-tD_M^2}$ and $e^{-tD_F^2}$ are both completely positive then for $D^2=1\tsr D_M^2 + D_F^2\tsr 1$,  $e^{-tD^2}$ is as well. The converse holds when $e^{-tD^2}$ is conservative. \end{theorem}
 
\begin{proof} Because $1\tsr D_M^2$ and $D_F^2\tsr 1$ commute, therefore, $e^{-tD^2} = e^{-t(1\tsr D_M^2)}e^{-t(D_F^2\tsr 1)} = e^{-t(D_F^2\tsr 1)}e^{-t(1\tsr D_M^2)}$. Now suppose  $e^{-tD_M^2}$ and $e^{-tD_F^2}$ are completely positive. The tensor product of completely positive maps extends to a completely positive map with respect to the $\norm{\cdot}_{min}$ (see, for instance, \cite[Thm~12.3]{paulsen_cbmaps}; the standard result is for ungraded tensor product, but it applies since commutative \CStar-algebras are nuclear regardless of the grading and there's only one cross norm across both settings). Since $\Ac_M$ is nuclear, $1\tsr e^{-tD_M^2}, e^{-tD_F^2}\tsr 1$ are completely positive on $\Ac_F\tsr \Ac_M = \Ac_F\tsr_{min} \Ac_M$. Furthermore, $1\tsr D_M^2 , D_F^2\tsr 1$ commute, and $e^{-tD^2}$ is composition of completely positive maps and also completely positive.
 
When $H_t:=e^{-tD^2}$ is unital, $H_t(1\tsr \Ac_M) = 1\tsr e^{-tD_M}(\Ac_M)$.  Since $1\tsr \Ac_M$ generates the  \CStar-algebra, $\KK\tsr_\KK\Ac_M \isomorph \Ac_M$, with $\KK=\CC, \RR$ depending on the underlying Hilbert space, the result follows for $e^{-tD_M^2}$, with symmetric argument for $e^{-tD_F^2}$.
\end{proof}
 
If $e^{-tD_M^2}$ and $e^{-tD_F^2}$ are contractive, then the composition $e^{-tD^2}$ is contractive as well. However, the converse does not hold. The same applies to conservativeness: if $e^{-tD_M^2}$ and $e^{-tD_F^2}$ are multiplication by $a\neq 0,1$ and $1/a$ respectively then the composition is identity while neither map is conservative and only one is contractive. 
 
Notice that since $D_F^2$ is a bounded operator, it generates a completely positive semigroup if and only if it's completely conditionally positive; in particular, the general form of for general form for generators of completely positive uniformly continuous semigroups, i.e., bounded generators, is known (see, for instance \cite{alicki_qds}).  
 
\begin{example}\label{remark_cacic_motivating_dirac}
On the product spectral triple, $M\times F = (\C^\infty(M)\tsr A_F,L^2(M, S\tsr H_F
), D_M\tsr 1 + \gamma_M\tsr D_F; J_M\tsr J_F , \gamma_M \tsr \gamma_F)$, it's assumed that $D_F$ does not know about $M$. This can be generalized slightly following \cite{cacic_reconstruction} to the picture where $M\times H_F\to M$ is a trivial bundle with a trivial connection $\grad^F$ and $H=S\tensor (M\times H_F)$ is a twisted spinor bundle, i.e., the Clifford action takes place on $S$. Noting observation~\ref{obv_dirac_perturbation_form}, the geometric Dirac operator on $H$ is given by \begin{equation}\label{eq_perturbed_dirac_op}
 \diracop_H:= \diracop_S \tsr 1 + c\tsr \grad^F
\end{equation}
where $c$ is the Clifford action. \namecite{cacic_reconstruction} defines the operator $D=\slashed{D}_H + 1\tsr D_F$ 
. $D$ is Dirac type operator on the spectral triple $(\C(M, A), L^2(M, H), D)$ where $A:=\Ls\tsr (X\times A_F)$ for $\Ls$ a real, unital, trivial sub-bundle of $\End(S)$ given by $\Ls_x:=\RR 1_{S_x}$. This follows easily as $[D, f] = [\diracop_H, f] + [1\tsr D_F, f]$ and $[1\tsr D_F, f]$ vanishes as $f\in \C^\infty(M)$ on each fiber is scalar multiplication that commutes with the $D_F$ which on the fibers is a matrix of scalars.
\begin{remark}
Note that the fibers being $\RR 1_{S_x}$ is just a specialization of the theorem~\ref{thm_commutes_with_cliffod_action}.
\end{remark}
In the bundle $H$, the symmetric operator $D_F$ on the fibers ${(H_F)}_{m\in M}$ can now vary with $m\in M$. Note the mixed term $(\diracop\circ c + c\circ\diracop) \tsr (\grad^F + \Omega_F)$ that now appears even when $D_F$ is constant. Complete positivity on such bundles is addressed  with same geometric methods as general almost commutative triples and spinor bundles. 
\end{example}
 
For the spinor bundle $S\to M$, we want to consider the complete positivity\footnote{\namecite{davies_cstar_bundles} considers a \CStar-bundle structure on $\C(M, \End(E))$ over a Riemannian manifold $(M, g)$ by defining the involution and $\CStar$-algebras pointwise then integrating to get the global structure. Such bundles are isomorphic to trivial bundles and do not capture the noncommutative geometry perpective. The complete positivity of heat-semigroups on the Clifford \CStar-bundles is considered in \cite{cipriani_riemannian,davies_cstar_bundles} and similar results on twisted spinor bundles can be derived using techniques considered here.} of semigroup generated by the heat operator $e^{-d\diracop^2}$ on an appropriate algebra $\Ac\subset \Bc(L^2(M, \End(S)))$. The algebra $\Ac$ will contain the Hilbert-Schmidt operators on $L^2(M, S)$, with Hilbert-Schmidt inner product $(f, g)_{HS}$, \begin{equation}
(f, g)_{HS} := \Tr_{HS}(fg^*) =  \sum_i \ip{e_i, fg^*e_i}_{L^2(M, E)}= \sum  \int(e_i, f(x)g(x)^*e_i)_{E,x}\dVol(M)
\end{equation}
where $(e_i)$ is an orthonormal system for $L^2(M, E)$. Such systems are provided by self-adjoint elliptic operators on the bundle: if $P$ is self-adjoint elliptic operator $P:\MSect(E)\to \MSect(E)$, then eigenspaces of $P$, $E_\lambda:=\ker(P-\lambda \one)$, are finite dimensional, consist of smooth sections, and give a complete orthonormal system for $L^2(E)$, $L^2(E)=\oplus_\lambda E_\lambda$. Additionally, for an elliptic operator $P:\MSect(E)\to \MSect(E)$ of order $m$ on vector bundle $E$ over compact $X$, on any open set $U\subset X, u\in L^2_s(E)$ where $L^2_s(E), s\in \RR$ is the Sobolev space, $Pu\big|_U\in \C^\infty$ implies $u\big|_U\in \C^\infty$. Now the connection laplacian $\grad^*\grad$ is an elliptic operator. By \cite[Thm~3.7]{bandara_density}, the closure of the connection laplacian of $E$, $\opclosure{\laplace}^E$ is self-adjoint. Since $\opclosure{\laplace}^E$ restricts to $\laplace^E$ over ${\MSect^\infty(E)}$, and the eigenspaces consist of smooth sections, we have a basis for $L^2(E)$ in terms of smooth eigensections of $\laplace^E$ (see \cite[Thm~III.5.2,III.5.8,Def~III.2.3]{lawson_spin}).

$\Tr_{HS}$, being lower semicontinuous and faithful is permissible in the sense of \namecite{albeverio_markov_cstar}; this means we can use to noncommutative Dirichlet form theory to consider the question of generating completely Markov semigroups; we introduce this next.  
  
\subsection{Noncommutative Dirichlet forms}\label{noncommutative_dirichlet_forms}
 
Recall from \cite{albeverio_markov_cstar}, for a \CStar-algebra $\Ac$ with a lower semicontinuous faithful trace $\tau$, $L^2(\Ac, \tau)$ is the completion on the pre-Hilbert space $\{x: \tau(x^*x)<\infty \}$ with inner product $\ip{x, y}_\tau:=\tau(y^*x)$. Set $L_h^2(\Ac, \tau):=\{x\in L^2(\Ac, \tau):x=x^*\}$.
 
\begin{definition} A strongly continuous contraction semigroup on $L^2(A, \tau)$ is symmetric if for all $x, y$,  $\ip{\Phi_t(x), y } = \ip{x, \Phi_t(y)}$. Further, if $0\leq\Phi_t(x)\leq 1$ whenever $0\leq x\leq 1$ then the semigroup is Markov.
\end{definition}
 
\begin{definition}\label{def_dirichlet_form} Suppose $\Ec(x, x)$ is a closed, quadratic form on $L_h^2(\Ac, \tau)$, with dense domain $\Dom(\Ec)$ with $f(\Dom(\Ec)) = \Dom(\Ec)$ for $f\in\txt{Lip}(\RR, 0)$, the Banach space of Lipshitz continuous functions that fix zero, $\norm{f}_{\txt{lip}}:=\inf\{m:|f(x)-f(y)|\leq m|x-y|\txt{ for all } x,y\in \RR\}$. Then $\Ec$ is a Dirichlet form if $\Ec(f(x), f(x))\leq \norm{f}^2_{\txt{lip}}\Ec(x, x)$. The form $\Ec$ is completely Dirichlet if $\Ec\tsr \one_n$ is Dirichlet for each $n\in \NN$.
\end{definition}
 
For a symmetric Markov semigroup  $(\Phi_t)_{t\geq 0}$ on $L^2(A, \tau)$, with a positive self-adjoint generator $\Lc$ on $L^2(A, \tau)$, $\Phi_t=e^{-t\Lc}$, the associated quadratic form\cite[thm~2.7]{albeverio_markov_cstar} is given by $$\Ec^\Lc(x):=\Ec^\Lc(x, x)=\ip{\Lc^{1/2}x, \Lc^{1/2}x}=\norm{\Lc^{1/2}x}_{L^2(\Ac, \tau)}^2
$$ \begin{theorem}(\cite[Thm~2.7,3.2]{albeverio_markov_cstar})
Dirichlet forms are in one--one correspondence with symmetric Markov semigroups: the positive quadratic  form $\Ec^\Lc$ associated to the generator $\Lc$ for a symmetric Markov semigroup $\Phi_t$ is a Dirichlet form. And conversely, if $\Ec^\Lc(x, x)$ is a Dirichlet form on $L^2_h(A, \tau)\subset L^2(A, \tau)$ then $H$ generates a Markov semigroup on $L^2(A, \tau)$. This extends to complete Markovity: $\Ec^\Lc$ is completely Dirichlet if and only if $\Lc$ generates a completely Markov semigroup.
\end{theorem}
 
 \section{Complete Markovity on spinor bundles}\label{sec_markov_spinor}
 
\subsection{Twisted spinor bundles}\label{sec_twisted_laplace}
 
To examine complete positivity on almost commutative spectral triples we will need the Bochner, Weitzenb\"{o}ck, and Lichnerowicz–Schr\"{o}odinger identities from spin geometry. Suppose $H$ is a twisted spinor bundle, with Dirac laplacian, $\dlaplace_H$, while the associated connection laplacian is $\laplace^H$. The Riemann curvature tensor, $\Rm\in \Sect {\tsr^{[4]}T^*M}$, is defined by $\Rm(v_1,v_2,v_3,v_4) = \ip{R(v_1,v_2)v_3,v_4},  v_i \in \Vect(X)$, $R_m(v_1, v_2)\in \Endo{T_m(X)}, m\in X$ for each $v_1, v_2\in T_mM$ defined by $
R_m(v_1, v_2)v = (\grad_{v_1}\grad_{v_2}v-\grad_{v_2}\grad_{v_1} v-\grad_{[v_1,v_2]}v)(m)\txt{ for }v_i \in \Vect(M)
$. The symmetries of the curvature imply that $\hat R$ can be viewed as an operator, $
\hat R:\Lambda^2TM\tsr\Lambda^2TM\to \RR
$. For any vector bundle $E$ with connection $\grad^E$, $R^E_{v_1, v_2}$ will denote the curvature transformation of the bundle, $R^E_{v_1, v_2}:\MSect(E)\to \MSect(E),  e\to (\grad_{v_1}\grad_{v_2} - \grad_{v_2}\grad_{v_1} -\grad_{[v_1, v_2]})e\in \MSect(E)$.

\begin{definition} For the connection laplacian $\laplace=\grad^*\grad$ and the Dirac operator $D$ for any any Dirac bundle $S$ over $X$, $n=\dim X$, with $R^S_{u,v}$ the curvature transformation of $S$, $(e_i)$ the orthonormal tangent frame, the \emph{general Bochner identity} states  \begin{equation}\label{eq_wietzenbock_identity}
D^2=\laplace + \mathfrak{R} \txt{~where~} \mathfrak{R}(\phi):= \dfrac{1}{2}\sum_{j,k\in[n]} e_j\cdot e_k\cdot R^S_{e_ie_j}(\phi)
\end{equation}
 
This specializes to the \emph{Lichnerowicz–Schr\"{o}odinger formula} when $X$ is a spin manifold, $S$ the bundle of spinors over $X$ with the canonical Riemannian connection, $D$, the Atiyah-Singer operator,
where $\kappa:X\to\RR$ is the scalar curvature, $\kappa=-\sum_{j,k\in[n]}\ip{R^S_{e_i,e_j}(e_i), e_j}$, \begin{equation}\label{eq_lichnerowicz_identity}
D^2=\laplace + \dfrac{1}{4}\kappa
\end{equation}
 
On any twisted spinor bundle $\Sc\tsr \Wc$, with Dirac operator $D_{\Sc\tsr\Wc}^2$, the connection laplacian $\laplace^{\Sc\tsr \Wc}$, the general Bochner identity becomes \emph{Bochner-Weitzenb\"{o}ck identity} \begin{equation}\label{eq_twisted_lichnerowicz_identity}
D_{\Sc\tsr\Wc}^2 = \laplace^{\Sc\tsr \Wc} + \dfrac{1}{4}\kappa + \mathfrak{R}^{\Wc} \txt{~where~}\mathfrak{R}^{\Wc}(\sigma\tsr w) = \dfrac{1}{2}\sum_{j,k\in[n]}(e_je_k\sigma)\tsr R^\Wc_{e_j,e_k}w
\end{equation}
\end{definition}

Suppose $H=\Sc\tsr E$, that is, $H$ is the spinor bundle $\Sc$ twisted by $E$. Using that the connection laplacian $\laplace^{\Sc\tsr E} = -\Tr((V, V')\to \grad^2_{V, V'})$ where $\grad_{V, V'}^2 = \grad_V\grad_{V'} -\grad_{\grad_VV'}$ which in the geodesic frame is becomes\begin{equation}\label{eq_laplacian_geodesic_coords}
    \laplace^{\Sc\tsr E}=-\sum_i\grad_{e_i}\grad_{e_i}
\end{equation} Explicitly the tensor connection laplacian is given by:
\begin{align}\label{eq_tensor_connection_laplacian}
\laplace^{\Sc\tsr E}\sigma &= -\sum_i \left( \grad_i^\Sc\tsr 1 + 1\tsr \grad_i^E \right)\left( \grad_i^\Sc\tsr 1 + 1\tsr \grad_i^E \right)\sigma\nonumber\\
&= -\sum_i\left( \grad^\Sc_i\grad^\Sc_i\tsr 1 + 2\grad_i^\Sc\tsr \grad_i^E + 1\tsr \grad_i^E\grad_i^E \right)\sigma = \left(\laplace^\Sc\tsr 1 - 2\sum_i\grad^\Sc_i\tsr\grad_i^E+ 1\tsr\laplace^E \right)\sigma
\end{align}
 
\noindent In general, to compute $e^{-t\laplace^{\Sc\tsr E}}$ Baker-Campbell-Hausdorff formula is needed as the the terms in the expansion of ${-t\laplace^{\Sc\tsr E}}$ don't commute.
A simple calculation that verifies that the terms commute when the curvatures of the bundles $E$ and $\Sc$ vanish identically --
 
\begin{lemma}\label{lm_grad_laplacian_commutator} For any vector bundle, the connection laplacian $[-\laplace^V, \sum_j\grad_j] =  \sum_{ij}R(i, j)\grad_i+ \sum_{ij}\grad_i R(i, j) $.
\end{lemma}
\begin{proof} In geodesic frame, $(e_i:i\in [\dim V])$, using $R(i, j)= \grad_i\grad_j-\grad_j\grad_i$, with shorthand $R(i, j):= R(e_i, e_j)$ \begin{align*}
-\laplace^V \sum_j\grad_j &= \sum_i \grad_i\grad_i\sum_j \grad_j\\
&= \sum_{ij} \grad_i\grad_i\grad_j = \sum_{ij}(\grad_i\grad_j\grad_i + \grad_i R(i, j)) = \sum_{ij}\grad_j\grad_i\grad_i + \sum_{ij}R(i, j)\grad_i+ \sum_{ij}\grad_i R(i, j)
\end{align*}
That is, $[-\laplace^V, \sum_j\grad_j] = \sum_{ij}R(i, j)\grad_i+ \sum_{ij}\grad_i R(i, j)$.
Using $R(i, j) = -R(j, i)$ and $R(i, i)=0$, $$
\sum_{ij}R(i, j)\grad_i = \sum_{i<j}\left( R(i,j)\grad_i + R(j, i)\grad_j\right) + \sum_{i=j}R(i,j)\grad_i = \sum_{i<j}\left( R(i,j)\grad_i + R(j, i)\grad_j\right) = \sum_{i<j}R(i,j)\left(\grad_i - \grad_j\right)
$$
By the second Bianchi identity, $(\grad_u R)(v, w) + (\grad_v R)(w, u) + (\grad_w R)(u, v) = 0
$, when $v=u=w$, $(\grad_u R)(u, u) = 0$, therefore  
$$
\sum_{ij}\grad_i R(i, j) = \sum_{i<j}(\grad_i R(i, j) + \grad_j R(j, i)) = \sum_{i<j}(\grad_i R(i, j) - \grad_j R(i, j)) = \sum_{i<j}(\grad_i-\grad_j)R(i, j)
$$
This yields $[-\laplace^V, \sum_j\grad_j] =  \sum_{ij}R(i, j)\grad_i+ \sum_{ij}\grad_i R(i, j)$

\end{proof}
 
\begin{corollary} If vector bundles $\Sc\to X, E\to X$ are such that the curvatures $R^E, R^\Sc$ satisfy $\sum_{ij}R^E(i, j)\grad^E_i+ \sum_{ij}\grad^E_i R^E(i, j) = 0$ and $\sum_{ij}R^\Sc(i, j)\grad^\Sc_i+ \sum_{ij}\grad^\Sc_i R^S(i, j) = 0$
then $$e^{-t\laplace^{\Sc\tsr E}} = e^{-t\laplace^{\Sc\tsr 1}}e^{t2\sum_i\grad_i^\Sc\tsr \grad_i^E}e^{-t\laplace^{1\tsr E}}$$
\end{corollary}
 
\begin{proof} Since $R^E, R^S$ satisfy the condition in  lemma~\ref{lm_grad_laplacian_commutator},  $\sum_i \grad^S_{i}\tsr \grad^E_{i}$ commutes with $1\tsr \laplace^E$ and $\laplace^\Sc\tsr 1$. As $\laplace^\Sc\tsr 1$ also commutes with $1\tsr \laplace^E$, so each term in the tensor laplacian commutes with the rest and one does not need Baker-Campbell-Hausdorff formula to compute the exponential. It follows that $
e^{-t\laplace^{\Sc\tsr E}} = e^{-t\laplace^{\Sc\tsr 1}}e^{-t(-2\sum_i\grad_i^\Sc\tsr \grad_i^E)}e^{-t\laplace^{1\tsr E}}
$
\end{proof}
 
The commutativity condition (equivalently the curvature condition in lemma~\ref{lm_grad_laplacian_commutator}) is satisfied more generally than for vanishing curvature. For instance, in section \ref{sec_canonical_connection}, it is verified that for canonical connection homogeneous space $K/H$ the connection acts through the Lie algebra and the connection laplacian for bundles associated to irredcuible representations of $H$, is the Casimir laplacian with an additive scalar, so it commutes with the connection. 

This corollary forces the semigroup generated by $\sum_i\grad_i^S\tsr \grad_i^E$ to be completely positive if the curvature conditions of lemma~\ref{lm_grad_laplacian_commutator} are satisfied and $\laplace_{S\tsr E}$ (and therefore, also $\laplace_{S\tsr 1}, \laplace_{1\tsr E}$) generates a completely positive semigroup\footnote{This can be viewed as the conditional complete positivity of $\sum_i\grad_i^S\tsr \grad_i^E$ on such bundles}.

\subsection{Complete positivity of heat semigroups}\label{sec_hilbert_schmidt_cp}
 
The quadratic forms associated to the laplacians are now considered using \CStar-Dirichlet form theory. As a warm up, the following result is in the spirit of \cite[Cor~4.4]{albeverio_markov_cstar}, however, $\laplace$ is essentially self-adjoint, not self-adjoint, so the argument is made directly from the definition, without appeal to formulation in terms of normal contractions on $\CStar$-algebras. Denoting by $\Hc(L^2(E))$ the Hilbert-Schmidt operators acting on $L^2(E)$ with $\Tr_{HS}$ inner product for any vector bundle, $E$, we have -- 
 
\begin{prop}\label{prop_laplacian_form_dirichlet_if_closable}
If the quadratic form, $\Ec_{\laplace}$, associated with the connection laplacian $\laplace_E$ on a vector bundle $E\to X$ with metric compatible connection, on $\Hc(L^2(E))$,  $\laplace_E$ is closed, then the form is Dirichlet and completely Dirichlet. The result also holds for $\laplace_E$ replaced by $\dlaplace_{E}$ for the spinor bundle $E$, and any positive operator $T=S^*S$.    
\end{prop}
 
\begin{proof} By definition~\ref{def_dirichlet_form}, with $\Ec_{\laplace}(x, x)=\Tr_{HS}(\laplace x^2)=\Tr_{HS}(x\laplace x)$ where $x=x^*\in , \Tr(x^2)<\infty$, we need to check that for $f\in\txt{Lip}(\RR, 0)$, $f(\Dom(\Ec)) = \Dom(\Ec)$ and $\Ec_{\laplace}(f(x), f(x))\leq \norm{f}^2_{\txt{lip}}\Ec_{\laplace}(x, x)$.
 
The condition $\Ec_{\laplace}(f(x), f(x))\leq \norm{f}^2_{\txt{lip}}\Ec_{\laplace}(x, x)$ follows by noting that $x$ and $x^2$ are compact and self-adjoint and, therefore, $x^2= \sum_i \alpha_i^2 P_i$ where $x= \sum_i\alpha_i P_i, \alpha_i\in \RR$ is the representation from the spectral theorem for compact self-adjoint operators. Now for $r\in \RR$, $f(r)/r\leq \norm{f}_{\txt{lip}}$ implying \aln{
\Ec_{\laplace}(f(x), f(x)) &= \Tr_{HS}(f(x)\laplace f(x)) = \sum_{i}\ip{e_i f(x), \grad^*\grad f(x) e_i} = \sum_i\norm{\grad f(x) e_i}^2\\ &\leq \norm{f}_{\txt{lip}}^2 \sum_i\norm{\grad x e_i}^2 = \norm{f}^2_{\txt{lip}} \Ec_{\laplace}(x, x) }
Since $y\in \Hc(L^2(E))$ can be written as $(y+y^*)/2 + (y-y^*)/2$, so to show invariance of the domain, it suffices to show $f(y)\in \Dom(\Ec_{\laplace})$ for $y\in\Dom(\laplace)$ self-adjoint. 
As $y\in \Dom(\laplace)$ means $\Tr_{HS}(y\laplace y) = \norm{\grad y}_{HS}<\infty$, $f(y)\in \Dom(\Ec_{\laplace})$ follows by same estimate. Therefore, if $\Ec_{\laplace}$ is closed, it's Dirichlet. 
 
Set ${\laplace}_n=\laplace\tsr \one_n$ for $\one_n$ the identity map on $\mat_n$. Since ${\laplace}_n = (\laplace^*\tsr \one_n)(\laplace\tsr \one_n)$, and any $y\in \mat_n^+$ is diagonalizable, the same analysis as before can be used. As $\one_n$ is closed, $\Ec_{{\laplace}_n}$ is closed if and only if $\Ec_{\laplace}$ is closed. This establishes the claim. The same argument applies to the Dirac laplacian $\dlaplace$ for the spinor bundle $E$ acting on $L^2(X, E)$ and for any $T$ of the specified form. \end{proof}

It remains to show that the form $\Ec_{\dlaplace}$ is closed on Hilbert-Schmidt operators on $L^2(E)$. We identify $\diracop$ with the Dirac operator extended to the $L^2$ sections, i.e. acting distributionally,  which is self-adjoint. $\Ec_{\dlaplace}$ is also identified with the extended version. Then by the Bochner identity and the fact that the curvature operator on compact manifold is self-adjoint and bounded, therefore closed, it also follows for the connection laplacian.      

\begin{theorem} Suppose $(E, h)$ is a spinor bundle over the compact Riemannian manifold $(X, g)$ with $D$ denoting the self-adjoint extension of the Dirac operator to $L^2(E)$. 
Let $\Hc$ be the Hilbert space of Hilbert-Schmidt operators on $L^2(E)$, $\ip{x, y}_{HS} = \Tr(x^*y)$. Then the quadratic form $\Ec_{\dlaplace}(x, y) = q(x, y):= \Tr(x^*yD^2)$ on $\Hc\times \Hc$ is closed, and therefore, $\Ec_{\dlaplace_H}$ is also completely Dirichlet. 
 \end{theorem}

\begin{proof} It's enough to show for the claim for $x,y$ self-adjoint, so we work with $q(x, x)=\Tr(x^2D^2)$). Let $(e_i)$ be a basis of $L^2(E)$ consisting of smooth eigensections of the laplacian $\laplace^E=\grad^*_E\grad_E$ associated to the connection for $D$. Note that $q$ is semibounded, since $\Tr(x^2D^2) = \sum_i\ip{D xe_i, D x e_i} = \norm{Dx}_{HS}\geq 0$ where we used that $x$ is self-adjoint and the trace is cyclic, so  $\Tr(x^2D^2)=\Tr(xD^2x)$.

Now suppose $(x_n)$ is a Cauchy sequence in norm $\norm{a}_+:= \sqrt{\norm{a}_{HS} + q(a, a)}$. So $(x_n)$ is Cauchy sequence in the Hilbert space $(\Hc,\norm{\cdot}_{HS} )$ implying $(x_n)\map{HS}x\in \Hc$. And $(x_n)$ being Cauchy in $\norm{\cdot}_+$ also gives that $q(x_n-x_m, x_n-x_m)\to 0$.  Because $q(a, a) = \norm{Da}^2_{HS}$, so $(Dx_n)$ is also Cauchy in $(\Hc, \norm{\cdot}_{HS})$ and therefore convergent with $\lim_{n \to\infty}Dx_n=g\in \Hc$. 

Now if for all $i, n$, $xe_i, x_ne_i\in \Dom(D)$, i.e., $xe_i, x_ne_i$ are weakly differentiable then we can show $g=Dx$, which gives $$
\lim_{n\to \infty }q(x_n-x, x_n-x)= \lim_{n\to \infty }\norm{D(x_n-x)}_{HS} =   \lim_{n\to \infty }\norm{Dx_n-g}_{HS}= 0$$
meaning the form $q$ is closed.

To see $g=Dx$ assuming $xe_i, x_ne_i$ are weakly $(L^2)$ differentiable, note that we have $x_n\map{HS}x$, so for all $i$, $x_ne_i\to xe_i$, then using that $D$ is self-adjoint, hence closed on $L^2(E)$, yields
\begin{align*}ge_i = \lim_{n\to\infty}(Dx_n)e_i=\lim_{n\to\infty}D(x_ne_i) = D(xe_i) &= (Dx)e_i
\end{align*}
Since $Dx$ and $g$ agree on the basis $(e_i)$, $Dx=g$. 

Finally, the weak differentiablity of $x_ne_i$ hold since  $e_i$ is smooth and $\norm{x_ne_i}^2\leq  \Tr(x_n^2)=\norm{x_n}_{HS}^2<\infty, \norm{Dx_ne_i}^2\leq  \Tr(x_n^2D^2)=\norm{Dx_n}_{HS}^2<\infty$, so $x_ne_i, Dx_ne_i\in L^2(E)$. And similarly $\norm{xe_i}\leq \norm{x}_{HS}$, $\lim_{n\to \infty}\norm{Dx_ne_i}\leq \norm{g}_{HS}$, so $xe_i\in \Dom(D)$.

Alternatively, the weak differentiability is also clear from noting that $x_ne_i$ is in the Sobolev space\footnote{We are using that the Sobolev space $W^{1, 2}(E)$ is the completion of smooth sections of $L^2(TX\tsr E)$ with norm $\norm{\cdot}_{W^{1, 2}}$, and for the connection $\grad$ on $E$ over compact $X$, $\closure{\grad}$, $W^{1, 2}(V) = \Dom(\closure{\grad})$ (see \cite[Cor~3.8]{bandara_density}). Also, see \cite[pg~116]{lawson_spin} for domain of closure of $D$} $W^{1, 2}(E)$ for all $i, n$ where we used that as $D\phi=\sum_kt^k\cdot \grad_{t_k} \phi$, $\ip{D\phi, D\phi} = \sum_{kl}\ip{t^k\cdot \grad_{t_k} \phi, t^l\cdot \grad_{t_l} \phi } = \sum_{kl}\delta_{kl}\ip{\grad_{t_k} \phi, \grad_{t_l} \phi }$. While convergence of $(x_n)$ in $\norm{\cdot}_+$ implies convergence of $x_ne_i$ in the norm $\norm{u}_{W^{1, 2}}:= \norm{u}_{L^2(E)} + \norm{\grad u}_{L^2(E)}$ for the $W^{1, 2}(E)$.  Hence $\lim _{n\to\infty}x_ne_i=xe_i\in W^{1, 2}$ (with $W^{1, 2} =\Dom(\closure\grad) = \Dom(D)$), and therefore is weakly differentiable and $xe_i\in \Dom(D)\subset L^2(E)$. \\
\end{proof}

\begin{corollary}
The form associated to laplacian, $\Ec_{\laplace}$, on the vector bundle $H$ is completely Dirichlet form on $\Hc(H)$. 
\end{corollary} 

Instead of using Bochner identity, one can also get at the result for the connection laplacian, $\laplace_E$, by adjusting the same reasoning to from $D$ to the closure $\closure{\grad}$ of $\grad$ using the results from \cite{bandara_density} after accounting for domain and codomain of $\grad$ not being the same Hilbert space as for $D$. This suggests that compactness of the underlying manifold is not essential. 
  
Kostant's cubic Dirac operator, $D^{1/3}$, is the Dirac operator associated to a linear combination of the canonical and Levi-Civita connection of the reductive space $K/H$. The laplacian, $(D^{1/3})^2$, can be expressed as the quadratic Casimir operator (i.e. the Casimir laplcian) with an additive scalar (see, for instance, \cite[Thm~3.3]{agricola2003connections}). By same argument it follows that it follows that it generates a quantum dynamical semigroup.
 
\begin{corollary} Kostant's cubic Dirac operator generates a quantum dynamical semigroup.
\end{corollary}
 
\noindent Similarly, quadratic form for a positive curvature operator, $\curveOp^H$ can be shown to be completely Dirichlet.
 
\begin{prop}
If $\curveOp\geq 0$ then the associated form, $\Ec_\curveOp$, is completely Dirichlet on $L^2(A, \tau)$. 
\end{prop}
 
\begin{proof}
First note that the $\curveOp^H$ at each fibre is bounded symmetric operator.  To see the symmetry, note it can immediately be checked that for any Riemannian connection, the curvature transformation is skew symmetric in the sense $\ip{R_{V, W}s, s'} = -\ip{s, R_{V, W}s'}$. Consider each term in  $\curveOp^H$, $ \ip{s, e_l\cdot e_k\cdot R_{e_{l}e_{k}} s'}$ for $s, s'\in \MSect(H)$, and an geodesic frame $(e_i)$, since $l\neq k$ must hold, $$
\ip{s, e_l\cdot e_k\cdot R_{e_{l}e_{k}} s'} = \ip{-R_{e_{l}e_{k}} e_k \cdot e_l\cdot s,   s'} = \ip{R_{e_{l}e_{k}} e_l \cdot e_k \cdot s,   s'} = \ip{e_l\cdot e_k\cdot R_{e_{l}e_{k}} s, s'}
$$
where to commute $e_l, e_k$ past $\grad_{e_l}, \grad_{e_k}$ inside $R_{lk}$, the product rule was used with the fact that the coordinates are geodesic, so covariantly constant. As $\curveOp^H$ varies smoothly, and the manifold is assumed to be compact, it's bounded globally. Since everywhere defined, symmetric operators are self-adjoint, and bounded operators are closed, $\curveOp^H$ is self-adjoint and closed. If $\curveOp$ is non-negative, $\curveOp^{1/2}$ exists and being bounded is closed, therefore, it follows as before that $\Ec_\curveOp$ is completely Dirichlet.
\end{proof}

Note that in the $L^2(A, \norm{\cdot}_{HS})$ setting the complete Markovity of the Dirac heat semigroup does not depend on the curvature unlike for \CStar-bundles where for Clifford bundles it does\cite{cipriani_riemannian}. 

So far we have considered $\Hc(L^2(E))$, equivalently, $L^2(\Ac, \tau)$ for $\tau=\Tr_{HS}, \Ac=\Bc(L^2(E))$. But $\Hc(L^2(E))$ is not unital which is needed to get at the dilations. So consider extension of a completely Markov semigroup $e^{-t\Lc}$ from $L^2(\Ac, \tau)$ to $\Ac$.
 
\begin{theorem}\label{thm_extending_from_hilbert_schmitt} The completely Markov semigroup $e^{-t\Lc}$ extends from $L^2(\Ac, \tau)$ to $\Bc(L^2(H))$ if and only if the $e^{-t\Lc}$ is completely Markov for each $t$ on the operator system\cite{paulsen_cbmaps}, i.e. a $*$-closed vector space, $\Oc(L^2(\Ac, \tau),1)$, generated by $L^2(\Ac, \tau)$ and $\one$.
\end{theorem}
 
\begin{proof} If $e^{-t\Lc}$ is not a completely Markov family of maps on $\Oc(L^2(\Ac, \tau),1)$ then obviously $e^{-t\Lc}$ does not extend to $\Bc(L^2(H))\supset \Oc(L^2(\Ac, \tau),1)$. If it's a completely Markov family, then as $\Oc(L^2(\Ac, \tau),1)$ is an operator system, so as completely positive maps, $e^{-t\Lc}$, extends to $\Bc(L^2(H))$ by Arveson's extension theorem\cite[Thm~7.5]{paulsen_cbmaps}. Complete Markovity follows since even though Hilbert-Schmidt operators are not norm dense, they are strongly dense in $\Bc(L^2(H))$.
\end{proof}
 
\begin{corollary} Suppose $\Lc(1) = 0$ then $e^{-t\Lc}$ is completely Markov on $\Oc(L^2(\Ac, \tau),1)$ and, therefore, on $\Bc(L^2(H))$.
\end{corollary}
 
\begin{proof} If $a\in\Oc(L^2(\Ac, \tau),1)$, then $a=\beta 1+\alpha$ with $\alpha \in L^2(\Ac, \tau)$, $\beta\in \CC$, and $\beta 1,\alpha$ commute. $e^{-t\Lc(\beta 1+\alpha)} = e^{-t\beta \Lc(1)}e^{-t\alpha} = e^{-t\alpha}$ which is completely Markov. The conclusion follows from the theorem~\ref{thm_extending_from_hilbert_schmitt}.
\end{proof}

\section{Evans-Hudson dilation  on reductive homogeneous spaces}\label{section_evans_hudson}
 
\subsection{Complete smoothness}
 
With the results on complete Markovity, the existence of the technical harness to handle covariance and the unboundedness of the generator remains to be checked. Recalling the setup from  \cite[Ch~8]{sinha_qsp}, we start with a \CStar-algebra $\Ac\subset \Bc(\Hs)$ on the Hilbert space $\Hs$, $G$ is a second countable, compact Lie group with finite dimensional Lie algebra, acting by a strongly continuous representation $G\ni g\to \alpha_g\in \Aut(\Ac)$ on $\Ac$.
\begin{definition}
Suppose $\{\chi_i:i\in [n]\}$ is the basis for the Lie algebra $\Lie{G}$, and $\D g$ the left Haar measure on $G$. The smooth algebra is defined by $\Ac_\infty=\{a:g\to \alpha_g(a)\txt{ is smooth  for all }g\in G\txt{ in norm topology}\}$. 
\end{definition}
 
Note that $\Ac_\infty=\cap_{k\in [n]} \Dom(\partial_k)$ where $\partial_i$ is closed $*$-derivation on $\Ac$ given by the generator of the automorphism group $(\alpha_{t\chi_i})_{t\in \RR}$. $\Ac_\infty$ can be equipped with Sobolev-type norms,\begin{equation}\label{eq_n_norm} \norm{a}_n=\sum_{i_1,i_2\dots i_k:k\leq n}\norm{\partial_{i_1}\cdots\partial_{i_k}(a)}\end{equation} with $\norm{a}_0=\norm{a}$. $\Ac_\infty$ is a Frechet algebra\footnote{The algebra $\Ac_\infty$ is also used in \cite[Pg~5]{farzad_twisted_spectral}; however, the norms $\norm{\cdot}_n$ are symmetrized explicitly.}

Let $T_t$ be a conservative semigroup with generator $\Lc$ that is covariant with respect to $G$, that is, $\alpha_g$ commutes with $T_t$ for all $t$ which is equivalent to $\Lc(\alpha_g(a)) = \alpha_g(\Lc(a))$ for all $a\in \Dom(\Lc)$. $\Lc$ is possibly be unbounded, but with $\Lc(\Ac_\infty) \subset \Ac_\infty\subset\Dom(\Lc)$. Additionally, suppose $\Lc$ is completely smooth, that is, a smooth map $\Lc$ between Frechet algberas, $\Mc_\infty,\Nc_\infty$ with respect to actions $\mu_g,\eta_g$ of compact Lie group $G$ on \CStar-agebras $\Mc, \Nc$ is completely smooth if there exists a constant $C$ and $p\in \ZZ^{\geq 0}$ satisfying for all $n,N\geq 0$ and $\xi\in \Mc_\infty \tsr \mat_n$, $$
\norm{\Lc\tsr 1_{\mat_N}(\xi)}_n\leq C\norm{\xi}_{n+p}
$$
Note that bounded operators are completely smooth since from equation~(\ref{eq_n_norm}), $\norm{\cdot}_l \geq \norm{\cdot}_q$ for all $l\geq q$.
Complete smoothness is a regularity condition that guarantees the convergence of the iterative scheme to construct the Evans-Hudson flow.
 
\begin{lemma}
Suppose $W_i$ is $w_i$-completely smooth for $i\in [N]$, then any polynomial in $W_i$'s is completely smooth of some order.
\end{lemma}
\begin{proof} First, since $W_i$ is $w_i$-completely smooth for $i\in [N]$, let $\norm{W_i\tsr 1_{\mat_N}(\xi)}_n\leq C_i\norm{\xi}_{n+w_i}$. By eq~\ref{eq_n_norm}, so we can assume $W_i$ are $w=\max(w_i)$-completely smooth, meaning $C=\max_{[N]}C_i$, $\norm{W_i\tsr 1_{\mat_N}(\xi)}_n\leq C\norm{\xi}_{n+w}$ for all $i$. This gives \aln{\norm{\sum_{i\in [N]}W_i\tsr 1_{\mat_N}(\xi)}_n\leq
\sum_{i\in [N]}\norm{W_i\tsr 1_{\mat_N}(\xi)}_n \leq NC\norm{\xi}_{n+w}
}
For $W_iW_j:=W_i\circ W_j$, $\norm{W_iW_j\tsr 1 \xi}_n = \norm{W_i\tsr 1(W_j\tsr 1) \xi}_n  \leq C_i\norm{W_j\tsr 1 \xi}_{n+w_i} \leq C_iC_j \norm{\xi}_{n+w_i+w_j}$
and the conclusion follows.
\end{proof}
 
\noindent We note the following version of \cite[Thm~8.1.28]{sinha_qsp} --
 
\begin{prop}\label{prop_poly_smooth} Suppose $\Lie{G}$ has basis ${X_i:i\in [m]}$, i.e. $X_i$'s generate one-parameter subgroups, then the $\Phi[X_i:i\in [m]]$ be a polynomial degree $p$ in $X_i$'s with coefficients in $\Bc(\Hs)$, which by the Lie algebra action on $\Ac_\infty$ defines a map $\Phi:\Ac_\infty\to \Ac_\infty$, then $\Phi$ is $p$-completely smooth. \end{prop}
 
\begin{proof} Set $\alpha$ as the norm of the largest coefficient of $\Phi$, wlog assume $\alpha \geq 1$. For any monomial $\Phi_i$ in $\Phi$, with $\xi=\sum_{[q]}x_l\tsr m_l, \Phi=\Phi[X_i:i\in[m]]$,\aln{
\norm{\Phi_i\tsr 1 \xi}_n &= \sum_{i_1\dots i_k, k\leq n}\norm{\left(\prod_{j\in [k]}X_{i_j}\tsr 1\right)\sum_{[q]}\Phi_i(x_l)\tsr m_l}\\
&\leq \alpha \sum_{i_1\dots i_k, k\leq  n+p}\norm{\left(\prod_{j\in [k]}X_{i_j}\tsr 1\right)\sum_{[q]}x_l\tsr m_l} = \alpha\norm{\Phi_i\tsr 1 \xi}_{n+p}
}
This yields $ \norm{\Phi\tsr 1 \xi}_n \leq N \alpha\norm{\xi}_{n+p}
$
where $\Phi$ has $N$ monomials.
\end{proof}
 
\begin{theorem}\label{existence_of_evans_hudson_dilation}(Existence of Evans-Hudson dilation\cite[Thm~8.1.38]{sinha_qsp}) If $(T_t)$ is a conservative quantum dynamical semigroup on a unital \CStar-algebra $\Ac$, covariant with respect to action of a second countable compact Lie group $G$,  with possibly unbounded generator $\Lc$ that is  $p$-completely smooth for some $p$ and $L(\Ac_\infty) \subset \Ac_\infty\subset\Dom(\Lc)$, then the Evans-Hudson dilation exists . 
\end{theorem}
 
By theorem~\ref{existence_of_evans_hudson_dilation}, the existence of Evans-Hudson dilation requires that the semigroup be conservative. As remarked before, this does not hold for the semigroups $e^{-tL}$, $L=\laplace, \dlaplace$.  While the commutation of the generator with the Lie group action and complete smoothness are tied to the Lie algebra structure. For homogeneous spaces, and to some extent, more generally, we work towards verifying -- and working around, the hypothesis needed.

\subsection{The endomorphism connection}

In requiring the semigroups to be conservative, one quickly notes that 
the laplacian $\laplace^H$ cannot be generate a conservative semigroup if it acts by composition on $\End(L^2(X, H))$
: fix a basis $(e_i)$ of eigensections of $\laplace^H$ for $L^2(H)$ and let $\lambda_i$ be eigenvalue for $\laplace_H$ on $e_i$, then $(e_i\tsr e_j^*)_{i,j}$ is a basis for  $\End(H)$. If $\laplace^H$ acts by composition on $\End(H)$ then it maps $e_i\tsr e_j^*$ to $\lambda_i e_i\tsr e_j^*$, but then $\laplace^H(1) = \laplace^H(\sum_ie_i\tsr e_i^*)= 0$ cannot hold. However, we have the following --

\begin{observation}\label{obv_identification_inside_endomorphism_bundle}
On the endomorphism bundle, the canonical connection $\grad^{\End(H)} = \grad^H\tsr 1 + 1\tsr \grad^{H^*}$ is easily seen define a conservative semigroup and is uniquely determined from $\grad^H$: if over $(U, \phi_U)$ the connection acts locally by $\grad(\sum \sigma^j\mu_j)=\sum_j(d\sigma^j)\mu_j + \sum_j\sigma^jA\mu_j$ for a matrix of $T^*M$-valued 1-forms $A$, 
then the dual connection acts with matrix  $\hat A:=-A^t$, 
and $\grad^{\End(H)}(\sum_{ij}\sigma^i_j\mu_i\tsr \mu^j)$ is given by \begin{align*}
\grad^{\End(H)}\sum_{ij}\sigma^i_j\mu_i\tsr \mu^j 
&= \sum_{ij}(d\sigma_j^i)\mu_i\tsr \mu^j +
\sum_{jk}[\sigma A - A\sigma]_{jk}\mu_k\tsr \mu^j
\end{align*}
Additionally, from previous computations, $\laplace^{\End(H)} = \laplace^{H}\tsr 1 + 2\sum_i\grad^H_i\tsr\grad^{H^*}_i +  1\tsr \laplace^{H^*}$, and therefore if $u^{H^*}\in\ker(\laplace^{H^*})$ is covariantly constant and non-zero, then  $\laplace^{\End(H)}(\phi\tsr u^{H^*}) = (\laplace^{H}\phi)\tsr u^{H^*}$, meaning $e^{-t\laplace^H}(\phi)$ on $H$ can be identified with  $e^{-t\laplace^{\End(H)}}(\phi\tsr u^{H^*})$ on $\End(H)$.
 
In terms of the acting on $\phi\tsr \psi$ this gives that $e^{-t\laplace^H}$ acting by composition on $\phi\tsr \psi$ i.e., $e^{-t\laplace^H}$ maps $\phi\tsr \psi$ to $(e^{-t\laplace^H}\phi)\tsr \psi$, can be identified with $U_{\psi} \circ e^{-t\laplace^{\End H}}(\phi\tsr u^{H^*})$, where $U_\psi$ is a change of basis for $H^*$ sending $u^{H^*}$ to $\psi$. Note all computations are in this basis and the identification between basis for $H^*$ and $H$ is no longer canonical. The existence of covariantly constant sections relates to holonomy. For homogeneous spaces with canonical connection such sections can be induced by using that the torsion and curvature tensors are covariant constant.    
\end{observation}
 
If $E$ is a hermitian (or euclidean) vector bundle with connection $\grad^E$ and $H$ a Dirac bundle with connection $\grad^H$ over $X$, then the $\phi\cAct (h\tsr e)\to (\phi\cAct h)\tsr e$ for $\phi\in \ClifB(X)$  defines a Clifford action on $H\tsr E$. The skew hermiticity of the action is obvious and as needed the tensor product connection $\grad^{H\tsr E}$ satisfies \aln{
\grad^{H\tsr E}(\phi \cAct (\sigma \tsr e)) 
&= (\grad^{\ClifB(X)}\phi)\cAct (\sigma\tsr e) + \phi\cAct\grad^{H\tsr E}\sigma\tsr e
}
 
 
Now the Dirac and Clifford structures are local as the Clifford multiplication acts on fibres and the connections can be computed in a chart. 
The local structures can then be glued to get the global structure. As a special case of tensor product bundles, consider $\End(H)$ for a Dirac bundle $H$. Suppose local sections $\mu_i:i\in[\dim H])$ form an orthonormal basis of $H$ in chart $(U, \phi_U)$, and the corresponding dual basis $(\mu^i)$ for $H^*$. Over the $U$, $\End\big|_{U}(H)$ is just the bundle $H\big|_{U}\tsr H^*\big|_{U}$ with the fibres given by $\Linspan{e_i\tsr e^j:i,j\in [\dim H]}$. This yields: 
if $H$ is a Dirac bundle, then $\End(H), H\tsr E$ are Dirac bundles as well.
Relevantly, there's the following observation -- 
 
\begin{prop} Semigroups generated by laplacians $\dlaplace_{\End{H}}, \laplace$ for the endomorphism connection are conservative. 
\end{prop}
\begin{proof} As $1=\sum_i \mu_i\tsr \mu^i$, $\grad(1)$ vanishes identically over $U$, and therefore, $\laplace(1) = 0$. This implies $\dlaplace(1) = 0$ as well.
\end{proof}
 
\begin{example}
The Clifford bundle can be viewed as the endomorphism bundle of the spinor bundle. The Clifford connection is naturally an endomorphism connection. It is a derivation on sections of the bundle, and, therefore, is zero on the identity element of the Clifford bundle.
\end{example}

\subsection{The canonical connection laplacian}\label{sec_canonical_connection}
 
Suppose the homogeneous space $M=K/H$ for compact, connected, Lie group $K$, closed Lie subgroup $H\subset K$ is reductive with $\Lie{K}=\Lie{H}\oplus \mathfrak{M}$ as a vectorspace for an $\txts{Ad}(H)$ invariant subspace $\mathfrak{M}$. $\mathfrak{M}$ is identified with $T_oM$ where $o=eH$ in the coset manifold $K/H$. The homogeneous space $K/H$ is principal $H$-bundle, $\pi: K\to K/H$ and carries a $K$ action. Note that if the $K$ acts effectively on reductive homogeneous space $K/H$ then $H$ is isomorphic to a subgroup of $\txt{GL}(\dim M, \RR)$, and the fiber bundle $\pi:K\to K/H$ is isomorphic to a sub-bundle of the principal frame bundle $F(M, \txt{GL}(\dim M, \RR))$. The $K$ action is assumed to be effective. We will consider homogeneous vector bundles over $K/H$, that is, a vector bundle $E\to K/H$ is such that $K$ acts on $E$, with $gE_x=E_{gx}$, and the action $g:E_x\to E_{gx}$ is an isomorphism for all $g\in K, x\in K/H$. $H$ induces automorphism at each fiber, meaning the fibers carry a representation of $E$.

Suppose additionally that $K$ is semisimple, so the Killing form $B_K$ defines a positive definite Riemannian metric on $K$ and an inner product on $\Lie{K}$ by $-B_K$  such that the reductive decomposition for $K/H$ satisfies $\mathfrak{M}=\Lie{H}^\perp$ with respect to $-B_K$. By left invariance of the Killing form, the inner product on $\Lie{G}$ extends to a Riemannian metric on $M=K/H$. Since the Lie group $K$ is compact and connected, the Lie algebra exponential agrees with the Riemannian exponential and is surjective. This means that Casimir laplacian commutes with action of both Lie group and the Lie algebra.
 
The connections of interest are invariant connections, where by the invariance of a connection under a diffeomorphism, $g:K/H\to K/H$ means $\grad_{g_*X}(g_*Y)=g_*(\grad_XY)$. There's a unique $K$-invariant connection in $K$ such that if $f_t=\exp(tX)$ be the 1-parameter subgroup of $K$ corresponding to $X\in \mathfrak{M}$ with a natural lift of $o$ to $u_o$ in the principal bundle, then the orbit of $\tilde f(u_o)$ is horizontal. More intuitively, connection $1$-form for the canonical connection is projection onto the $\Lie{H}$; the horizontal distribution is obtained at $o$ by translating $\mathfrak{M}$ by the left $K$-action.

The canonical connection\footnote{Referred to as canonical connection of second kind} is a metric connection, but is not necessarily torsion free, instead the torsion and curvature are parallel. The canonical connection, therefore, does not agree with the Levi-Civita connection unless $M$ is a symmetric homogeneous spaces.
Following \cites{meng_quantum_hall}{camporesi_homogeneous_spaces}\footnote{also, see \cite[\S~3.5]{friedrich_dirac} for the Dirac laplacian on Riemannian symmetric spaces} , we note the expressibility of the laplacian for the canonical connection in terms of the Casimir laplacian of $K$ and $H$ and as a Lie algebra action.
From this the complete smoothness of the canonical connection laplacian becomes evident.
\begin{lemma}\label{lem_canonical_conn_completely_smooth} The canonical connection laplacian for homogeneous vector bundle $E$ over $K/H$ is expressible in terms of Lie algebra action and is completely smooth.\end{lemma}
\begin{proof}

Take an orthonormal basis $(Y_i)$ for $\Lie{K}$ at $e\in K$ which contains an orthonormal basis $(X_i)$ for $\mathfrak{M}$.  As $K$ is compact, the Riemannian exponential agrees with Lie exponential, so the basis defines geodesic local coordinates at $K$. On $K/H$, $(X_i)$'s define geodesic local coordinates on $K/H$. Therefore, $\laplace^{E} = \sum \grad_{X_i}\grad_{X_i}$. In any local frame $(e_j)$ for $E$ over $K/H$ about $o$, since the connection 1-form, being projection onto the $\Lie{H}$, vanishes on $X_i$'s, $$\grad_{X_i} \left(\sum_j\phi_j e_j\right) = \sum_j X_i(\phi_j)e_j$$ implying $\laplace^{K/H} = \sum X^2_i$ locally. By $K$ invariance of the connection, this holds everywhere. Complete smoothness follows from proposition~\ref{prop_poly_smooth}.
\end{proof}
 
\noindent Noting the representations of $K, H$ at play here gives the connection laplacian in terms of the Casimir operators of the Lie groups --
 
\begin{prop} For a homogeneous vector bundle $E$ over $K/H$ with the canonical connection, $$\laplace^{E} = -C_2(K, \MSect(E)) + C_2(H, E)
$$
where $C_2(K, \MSect(E))$ and $ C_2(H, E)$ are Casimir operators for $K$ and $H$, the representation for $K$ being the induced representation on $\MSect(E)$ and the representation of $H$, the representation on the fibers defining $E$. The action of $H$ on the sections is pointwise.
\end{prop}
 
\begin{proof} Proof is contained in lemma~\ref{lem_canonical_conn_completely_smooth}, since  $\laplace^{K/H} = -\sum X^2_i$  and $X_i$'s for a basis for $\mathfrak{M}$ while the remaining $(Y_i)$'s give the Casimir operator for $H$.      
\end{proof}
 
If the representation of $H$ is irreducible, then $C_2(H, E)$ is multiplication by a constant. If not, then on each fibre it's a bounded operator, and therefore, on $\MSect(E)$ it's a bounded operator as $E$ is finite dimensional, and $K/H$ compact.  This construction can be pushed to the endomorphism bundle.
 
\begin{observation} If $\phi_g$ is the identification between fibers $E_x, E_{gx}$ then coupling the duality between the fibers $E^*_y, E_y$ with $\phi_g$ gives the identification between $E^*_x, E_{gx}^*$. Therefore, the dual bundle, $E^*$, and similarly the endomorphism bundle $E\tsr E^*$, are also homogeneous.  $E^*, E\tsr E^*$ are defined by $\rho_E^*, \rho_E\tsr \rho^*_E$ where $\rho_E:H\to \End(E)$ is representation defining $E$ and $\rho_E^*$ the dual representation. $G$ has an induced representation on $\MSect(E\tsr E^*)$, so the canonical connection laplacian on $E\tsr E^*$ can be expressed similarly -- 
$$\laplace^{E\tsr E^*} = -C_2(K, \MSect(E\tsr E^*)) + C_2(H, E\tsr E^*)
$$
Since $C_2(H, E\tsr E^*)$ is a bounded operator, and $\laplace^{E\tsr E^*}$ is a polynomial in Lie algebra action and, therefore, completely smooth.
\end{observation}

\subsection{Quantum stochastic dilation on homogeneous spinor bundles}
Throughout this section, $E, S$ are homogeneous vector bundle with the canonical connection on reductive homogeneous space $M:=K/H$, $K$ compact, $H$ closed. The metric is the bi-invariant metric from the Killing form. $S$ is a homogeneous Clifford module bundle, i.e. a homogeneous twisted spinor bundle. Note the compatibility of the Clifford action with the homogeneous structure of the bundle, from which the $K$-invariance of the curvature operator will follow --
 
\begin{lemma} The Clifford action $c:T^*M\to \End(S)$ commutes with the $K$-action. 
\end{lemma}
\begin{proof} We first check that Clifford multiplication $\ClifA(T^*M)$ commutes with $K$-action so the Clifford action is well defined. Let $\phi_g$ be the vector space isomorphism given by the left action of $g\in K, \phi_g(s):=g\cdot s$ on $E$ between $E_x, E_{gx}$. By invariance of the metric an orthonormal basis of $(e_i)$ at $T^*_xM$ is mapped to an orthonormal basis under $\phi_g$ and this gives an identification ${\ClifA_{x}}$ between ${\ClifA_{gx}}$. Denoting Clifford multiplication at $x, gx$, by $\cdot_{\ClifA_{x}}, \cdot_{\ClifA_{gx}}$, since $\ip{\phi_g(u), \phi_g(v)}_{gx} = \ip{u, v}_x$,  $\phi_g:T^*_xM\to \ClifA(T^*_{gx}M)$ satisfies the universal property for Clifford algebras $\phi_g(u)\cdot_{\ClifA_{gx}} \phi_g(u) = \ip{u, u} 1_{\ClifA(T^*_{gx}M)}$, and therefore, $\phi_g$ extends to an algebra isomorphism $\phi_g:\ClifA(T^*_{x}M)\to \ClifA(T^*_{gx}M)$. This means  $\phi_g(u\cdot_{\ClifA_{x}} v) = \phi_g(u)\cdot_{\ClifA_{gx}} \phi(v)$ holds, and the Clifford multiplication commutes with the $K$-action.    
 
Now consider the Clifford action $c:T^*M\to \End(S)$. Since Clifford action on a twisted spinor $S:=\Sc\tsr \Wc$ bundle only acts on the spinor bundle piece $\Sc$, it suffices to verify that for any $v\in \Sc, \phi_g(\Psi(v)) = (\phi_g(\Psi))(\phi_g(v))$ where $(\phi_g(\Psi))$ is the action of $g$ on the $\Psi\in \End(\Sc)$. By linearity of $K$-action, wlog assume $\Psi=f_i\tsr f_j^*$ for a basis $(f_i)$ of $\Sc$, then $(\phi_g(\Psi)) = (\phi_g(f_i))\tsr (\phi_g(f_j^*))$. Since $\phi_g(f_j^*) = f_j^*\circ \phi_{g^{-1}}$ (equivalently, these can be explictly written in terns of the defining representation and its dual), $\phi_g(\Psi(v)) = (\phi_g(\Psi))(\phi_g(v))$ follows.     
\end{proof}

\begin{lemma}The the curvature operator $\mathfrak{R}^S$ from the general Bochner identity (eq~\ref{eq_wietzenbock_identity}) is completely smooth and commutes with the Lie group action of $K$, i.e., it's $K$-covariant.
\end{lemma}
\begin{proof} As $K$ is compact, $H$ closed, $K/H$ is compact, and therefore, the curvature operator is bounded and completely smooth. If $S$ is the spinor bundle, and the canonical connection is the Levi-Civita connection, i.e. $K/H$ is a symmetric space, then by equation~(\ref{eq_lichnerowicz_identity}), the $\mathfrak{R}^S$ is multiplication by scalar curvature $\kappa$, and its $K$-invariance is clear. More generally, the canonical connection is $K$-invariant, so by definition the curvature transformation $R^S$ commutes with the action of $K$. And as $
\mathfrak{R}^S(\phi):= \tfrac{1}{2}\sum_{j,k\in[n]} e_j\cdot e_k\cdot R^S_{e_ie_j}(\phi)
$, using that the Clifford multiplication commutes with action of $K$, the invariance holds.
\end{proof}
 
\begin{example}(\cite[Ch~3]{friedrich_dirac}) An example where complete smoothness and the covariance are particularly transparent is when $K/H$ is a symmetric space. The Dirac laplacian $\dlaplace_{K/H}$ for the canonical connection is given by $\dlaplace = \Omega_K + \kappa/8$ where $\kappa$ is the scalar curvature and $\Omega_K$ the Casimir operator for $K$. 
\end{example}
 
\begin{theorem} Evans-Hudson dilation exists for the semigroup $(T_t)_{t\geq 0}$ generated by the canonical connection laplacian, $\laplace^{\End{E}}$, on $\End{E}$. \end{theorem}
 
\begin{proof} As established $\laplace^{\End{E}}$ is  completely smooth and commutes with action of $K$, and $(T_t)_{t\geq 0}$ is a conservative quantum dynamical semigroup, the conclusion follows from theorem~\ref{existence_of_evans_hudson_dilation}.
\end{proof}

\begin{corollary}
Evans-Hudson dilation exists for the semigroup $(T_t)_{t\geq 0}$ generated by the Dirac laplacian $\dlaplace_{\End{S}}$.
\end{corollary}
\begin{proof} Since $\grad^{\End(S)}(1)=0$, $\dlaplace_{\End(S)}(1)=0$. The result for $\dlaplace_{\End{S}}$ is immediate from $K$-covariance and complete smoothness of $\mathfrak{R}^S$ using the general Bochner identity.
\end{proof}

\begin{remark} The only obstruction to pushing the arguments for Kostant's cubic Dirac laplacian is the conservativeness of the heat semigroup. Though, the same identification with the endomorphism connection can again be made, and the ``endomorphism cubic Dirac operator'' can be expanded in terms of Kostant's cubic Dirac laplacian and the connection laplacian for the dual bundle as in section~\ref{sec_twisted_laplace}.
\end{remark}
 
On almost commutative spectral triples, where the Dirac operator is a perturbation of the geometric Dirac laplacian, the endomorphism trick, along with observation~\ref{obv_identification_inside_endomorphism_bundle} finds use again. The following example considers the existence of Evans-Hudson dilation with respect to the perturbed geometric Dirac laplacian on the endomorphism bundle. 
 
\begin{example} Consider the almost commutative spectral triple, $\mathfrak{A}_{K/H}:=(\C^\infty(K/H, A), L^2(K/H, \Hs), D_0)$. $\Hs$ being a bundle of Clifford modules is a twisted spinor bundle $\Hs:=\Wc\tsr \Sc$. Let $\diracop$ be the geometric Dirac operator on $\Wc\tsr \Sc$ for the canonical connection on $K/H$, with $D_0=\diracop+B$ for some endomorphism $B$, where the Dirac operators are acting distributionally on $L^2(K/H, \Hs)$.
If the perturbation $B$ is self-adjoint and bounded, $D_0^2=(\diracop+B)^2$ is easily seen to be completely Dirichlet: it's just the Dirac laplacian for a connection with with a different potential on the twisting space. Completely smoothness is also immediately clear by proposition~\ref{prop_poly_smooth}. To handle conservativeness, we pass to the endomorphism bundle $\End(\Hs) = \End(\Wc)\tsr \End(\Sc)$ (where topological closure is implicit). 
%
Since $D_{0,\End(\Wc)}$ acts locally by commutator associated to potential $\Omega_B$ for the endomorphism $B$, the semigroup is conservative, and as before, the existence of Evans-Hudson dilation follows if the potential is covariant with respect to the action of $H$. An ensemble of Dirac operators, as for fuzzy spectral triples\cite{hessam2022noncommutative}, can be realized by randomizing the connection potential, but now it also has a geometric interpretation through the Bochner identity.

\end{example}

\printbibliography

@article{albeverio_stochastic_quantization,
  title={Grassmannian stochastic analysis and the stochastic quantization of Euclidean Fermions},
  author={Albeverio, Sergio and Borasi, Luigi and De Vecchi, Francesco C and Gubinelli, Massimiliano},
  journal={Probability Theory and Related Fields},
  pages={1--87},
  year={2022},
  publisher={Springer}
}

@article{hessam2022noncommutative,
  title={From noncommutative geometry to random matrix theory},
  author={Hessam, Hamed and Khalkhali, Masoud and Pagliaroli, Nathan and Verhoeven, Luuk S},
  journal={Journal of Physics A: Mathematical and Theoretical},
  volume={55},
  number={41},
  pages={413002},
  year={2022},
  publisher={IOP Publishing}
}

@article{barrett2015matrix,
  title={Matrix geometries and fuzzy spaces as finite spectral triples},
  author={Barrett, John W},
  journal={Journal of Mathematical Physics},
  volume={56},
  number={8},
  pages={082301},
  year={2015},
  publisher={AIP Publishing LLC}
}

@article{goswami2000stochastic,
  title={Stochastic dilation of a quantum dynamical semigroup on a separable unital C* algebra},
  author={Goswami, Debashish and Pal, Arup Kumar and Sinha, Kalyan B},
  journal={Infinite Dimensional Analysis, Quantum Probability and Related Topics},
  volume={3},
  number={01},
  pages={177--184},
  year={2000},
  publisher={World Scientific}
}

@article{khalkhali_second_quant,
  title={Second quantization and the spectral action},
  author={Dong, Rui and Khalkhali, Masoud and van Suijlekom, Walter D},
  journal={Journal of Geometry and Physics},
  volume={167},
  pages={104285},
  year={2021},
  publisher={Elsevier}
}

@article{farzad_twisted_spectral,
  title={On the Chern-Gauss-Bonnet theorem and conformally twisted spectral triples for \CStar-dynamical systems},
  author={Fathizadeh, Farzad and Gabriel, Olivier and others},
  journal={SIGMA. Symmetry, Integrability and Geometry: Methods and Applications},
  volume={12},
  pages={016},
  year={2016},
  publisher={SIGMA. Symmetry, Integrability and Geometry: Methods and Applications}
}

@article{davies_cstar_bundles,
  title={Markov semigroups on C*-bundles},
  author={Davies, E Brian and Rothaus, OS},
  journal={Journal of functional analysis},
  volume={85},
  number={2},
  pages={264--286},
  year={1989},
  publisher={Elsevier}
}

@inbook{Hudson_quantum_stochastic_flows,
author = {R.L. Hudson},
editor = {B. Grigelionis and Yu. V. Prohorov and V. V. Sazonov and V. Statulevičius},
%doi = {doi:10.1515/9783112314197-049},
%url = {https://doi.org/10.1515/9783112314197-049},
title = {Quantum stochastic flows},
booktitle = {Vol. 1},
year = {2020},
publisher = {De Gruyter},
pages = {512--525},
lastchecked = {2022-08-18}
}

@incollection{parthasarathy_diffusions,
  title={Markov chains as Evans-Hudson diffusions in Fock space},
  author={Parthasarathy, Kalyanapuram Rangachari and Sinha, Kalyan B},
  booktitle={S{\'e}minaire de Probabilit{\'e}s XXIV 1988/89},
  pages={362--369},
  year={1990},
  publisher={Springer}
}

@article{riemannian_feynman_kac,
  title={The proof of the Feynman--Kac formula for heat equation on a compact Riemannian manifold},
  author={Obrezkov, Oleg O},
  journal={Infinite Dimensional Analysis, Quantum Probability and Related Topics},
  volume={6},
  number={02},
  pages={311--320},
  year={2003},
  publisher={World Scientific}
}

@article{crismale_z2_cstar_tensor,
  title={On  C*norms on $\ZZ_2$-graded tensor products},
  author={Crismale, Vitonofrio and Rossi, Stefano and Zurlo, Paola},
  journal={Banach Journal of Mathematical Analysis},
  volume={16},
  number={1},
  pages={1--19},
  year={2022},
  publisher={Springer}
}

@article{albeverio_markov_cstar,
  title={Dirichlet forms and Markov semigroups on \CStar-algebras},
  author={Albeverio, Sergio and H{\o}egh-Krohn, Raphael},
  journal={Communications in Mathematical Physics},
  volume={56},
  number={2},
  pages={173--187},
  year={1977},
  publisher={Springer}
}

@article{bandara_density,
  title={Density problems on vector bundles and manifolds},
  author={Bandara, Lashi},
  journal={Proceedings of the American Mathematical Society},
  volume={142},
  number={8},
  pages={2683--2695},
  year={2014}
}

@article{camporesi_homogeneous_spaces,
  title={Harmonic analysis and propagators on homogeneous spaces},
  author={Camporesi, Roberto},
  journal={Physics Reports},
  volume={196},
  number={1-2},
  pages={1--134},
  year={1990},
  publisher={Elsevier}
}

@article{meng_quantum_hall,
  title={Geometric construction of the quantum Hall effect in all even dimensions},
  author={Meng, Guowu},
  journal={Journal of Physics A: Mathematical and General},
  volume={36},
  number={36},
  pages={9415},
  year={2003},
  publisher={IOP Publishing}
}

@inproceedings{belton_unbounded_quantum_flow,
  title={An algebraic construction of quantum flows with unbounded generators},
  author={Belton, Alexander CR and Wills, Stephen J},
  booktitle={Annales de l'IHP Probabilit{\'e}s et statistiques},
  volume={51},
  number={1},
  pages={349--375},
  year={2015}
  }

@inproceedings{farzad_brownian_bridge,
  title={Bell polynomials and Brownian bridge in spectral gravity models on multifractal Robertson--Walker cosmologies},
  author={Fathizadeh, Farzad and Kafkoulis, Yeorgia and Marcolli, Matilde},
  booktitle={Annales Henri Poincar{\'e}},
  volume={21},
  number={4},
  pages={1329--1382},
  year={2020},
  organization={Springer}
}

@article{agricola2003connections,
  title={Connections on naturally reductive spaces, their Dirac operator and homogeneous models in string theory},
  author={Agricola, Ilka},
  journal={Communications in mathematical physics},
  volume={232},
  pages={535--563},
  year={2003},
  publisher={Springer}
}

@book{meinrenken_clifford,
  title={Clifford algebras and Lie theory},
  author={Meinrenken, Eckhard},
  volume={58},
  year={2013},
  publisher={Springer}
}

@book{nualart2006malliavin,
  title={The Malliavin calculus and related topics},
  author={Nualart, David},
  volume={1995},
  year={2006},
  publisher={Springer}
}

@book{sinha_qsp,
  title={Quantum stochastic processes and noncommutative geometry},
  author={Sinha, Kalyan B and Goswami, Debashish},
  volume={169},
  year={2007},
  publisher={Cambridge University Press}
}

@book{paulsen_cbmaps,
  title={Completely bounded maps and operator algebras},
  author={Paulsen, Vern},
  number={78},
  year={2002},
  publisher={Cambridge University Press}
}

@book{parthasarathy_qsp,
  title={An introduction to quantum stochastic calculus},
  author={KR Parthasarathy},
  volume={85},
  year={1992},
  publisher={Springer Science \& Business Media}
}

@book{alicki_qds,
  title={Quantum dynamical semigroups and applications},
  author={Alicki, Robert and Lendi, Karl},
  volume={717},
  year={2007},
  publisher={Springer}
}

@book{friedrich_dirac,
  title={Dirac operators in Riemannian geometry},
  author={Friedrich, Thomas},
  volume={25},
  year={2000},
  publisher={American Mathematical Soc.}
}

@book{lawson_spin,
  title={Spin Geometry (PMS-38), Volume 38},
  author={Lawson, H Blaine and Michelsohn, Marie-Louise},
  year={2016},
  publisher={Princeton university press}
}

@book{nicolaescu_geometry,
  title={Lectures on the Geometry of Manifolds},
  author={Nicolaescu, Liviu I},
  year={2020},
  publisher={World Scientific}
}

@book{berline_heat,
  title={Heat kernels and Dirac operators},
  author={Berline, Nicole and Getzler, Ezra and Vergne, Michele},
  year={2003},
  publisher={Springer Science \& Business Media}
}

@book{kunita_stochastic_flows,
  title={Stochastic flows and jump-diffusions},
  author={Kunita, Hiroshi},
  year={2019},
  publisher={Springer}
}

@article{cacic_reconstruction,
  title={A reconstruction theorem for almost-commutative spectral triples},
  author={{\'C}a{\'c}i{\'c}, Branimir},
  journal={Letters in Mathematical Physics},
  volume={100},
  number={2},
  pages={181--202},
  year={2012},
  publisher={Springer}
}

@book{gracia_ncg,
  title={Elements of noncommutative geometry},
  author={Gracia-Bond{'\i}a, Jos{\'e} M and V{\'a}rilly, Joseph C and Figueroa, H{\'e}ctor},
  year={2013},
  publisher={Springer} % Science \& Business Media}

@article{cipriani_riemannian,
  title={Noncommutative potential theory and the sign of the curvature operator in Riemannian geometry},
  author={Cipriani, Fabio and Sauvageot, Jean-Luc},
  journal={Geometric \& Functional Analysis GAFA},
  volume={13},
  number={3},
  pages={521--545},
  year={2003},
  publisher={Springer}
}

@book{suijlekom,
  title={Noncommutative geometry and particle physics},
  author={Van Suijlekom, Walter D},
  year={2015},
  publisher={Springer}
}

@book{attal2006open,
  title={Open Quantum Systems I: The Hamiltonian Approach},
  author={Attal, St{\'e}phane and Joye, Alain and Pillet, Claude-Alain},
  year={2006},
  publisher={Springer}
}

@article{connes_spectral_entropy,
  title={Entropy and the spectral action},
  author={Chamseddine, Ali H and Connes, Alain and van Suijlekom, Walter D},
  journal={Communications in Mathematical Physics},
  volume={373},
  number={2},
  pages={457--471},
  year={2020},
  publisher={Springer}
}

\end{document}